\documentclass[twocolumn]{autart}
\usepackage{graphicx}
\usepackage{subfigure}
\usepackage{mathtools}
\usepackage{amsmath}
\usepackage{mathrsfs}
\usepackage{amssymb}
\usepackage{bm}
\usepackage{float}
\usepackage{hyperref}
\usepackage{wrapfig}
\usepackage{graphicx}
\usepackage{algorithm}

\usepackage{algorithmic}

\usepackage{picins}
\usepackage{color}
\newcommand{\pr}{\hbox{\sf P}}

\newcommand{\ep}{\hbox{\sf E}}

\newtheorem{asm}{Assumption}
\newtheorem{construction}{Construction}

\newtheorem{remark}{Remark}
\newtheorem{lemma}{Lemma}
\newenvironment{proof}{{\noindent \textbf{Proof}}}{\hfill $\square$\par}
\allowdisplaybreaks[4]
\begin{document}

\begin{frontmatter}

\title{A hybrid deep learning method for finite-horizon mean-field game problems}\vspace{-3em}

\thanks[footnoteinfo]{Corresponding author Jiaqin~Wei. Yu Zhang and Jiaqin Wei are co-first authors.}

\author[*]{Yu Zhang}\ead{52194404014@stu.ecnu.edu.cn},
\author[Z. Jin]{Zhuo Jin}\ead{zhuo.jin@mq.edu.au},
\author[*]{Jiaqin Wei}\ead{jqwei@stat.ecnu.edu.cn},
\author[G. Yin]{George Yin}\ead{gyin@uconn.edu}
\address[*]{Key Laboratory of Advanced Theory and Application in Statistics and
Data Science-MOE, School of Statistics, East China Normal University, Shanghai 200062, China}\vspace{-0.5em}
\address[Z. Jin]{Department of Actuarial Studies and Business Analytics, Macquarie University, 2109, NSW, Australia}\vspace{-0.5em}
\address[G. Yin]{Department of Mathematics, University of Connecticut, Storrs, CT 06269-1009, USA}\vspace{-1.5em}


\begin{abstract}                          
This paper develops a new deep learning algorithm to solve a class of finite-horizon mean-field games. The proposed hybrid algorithm uses Markov chain approximation method combined with a stochastic approximation-based iterative deep learning algorithm. Under the framework of finite-horizon mean-field games, the induced measure and Monte-Carlo algorithm are adopted to establish the iterative mean-field interaction in Markov chain approximation method and deep learning, respectively. The Markov chain approximation method plays a key role in constructing the iterative algorithm and estimating an initial value of a neural network, whereas stochastic approximation is used to find accurate parameters in a bounded region. The convergence of the hybrid algorithm is proved; two numerical examples are provided to illustrate the results.\vspace{-1.5em}
\end{abstract}

\end{frontmatter}

\section{Introduction}
The study of mean-field games (MFGs) is initiated independently in
the pioneering work of Lasry and Lions \cite{LjLp:06i,LjLp:06ii,LjLp:07,LjmLpl:07}
and Huang, Caines, and Malham\'e \cite{HmMrCp:06,HmCpMrp:07,HmCpeMr:07,HmCpMr:07}.
MFGs are multi-player games with weak interactions among players, which can be described as systems of two highly coupled nonlinear partial differential equations (PDEs); see \cite{BaFjYp:13,BaFjYs:15} and the references therein. The first equation, is the
Hamilton-Jacobi-Bellman (HJB) equation satisfied by the value function, whereas the second one represents the distribution given by a
Kolmogorov-Planck equation. Recently, the probabilistic methods, known as forward-backward stochastic differential equations, are also used to represent solutions of MFGs; see, e.g., \cite{CrDf:13} and \cite{CrLd:15}.

Usually, the solutions for MFGs cannot be solved in closed form.
As a result, numerical approximations are needed. In most of the existing
literature on numerical schemes for MFGs, often PDE-based approaches are used.
In \cite{AyCi:10}, the authors suggest a finite difference method for approximating the PDE system first, and they prove the convergence of the scheme in \cite{AyCfCi:13}. Lachapelle \textit{et al}. \cite{LaSjTg:10} construct computational procedures by utilizing the linear quadratic
nature, to approximate the value function, control, and
measure through the application of finite differences in the forward-backward
system. The study of Semi-Lagrangian schemes is conducted by Carlini
and Silva in \cite{CeSf:14} and \cite{CeAf:14}.
The Markov chain approximation method (MCAM) is used to construct approximations
for solutions of the MFGs with reflecting barriers in \cite{BeBaCa:18}.
In \cite{CjCdDf:19}, the authors propose an algorithm based on Picard iterations and the continuation method with the use of the master equation together with
smoothness assumptions on the infinite dimensional PDEs.

Although there have been numerous advancements,
developing numerical methods for solving MFGs, especially in high-dimensional problems has fallen behind. As is well known, the mesh-based
methods are prone to the curse of dimensionality, that is, their computational
complexity grows exponentially with spatial dimension. On the other
hand, the availability of powerful software packages like TensorFlow
and PyTorch in the public domain has enabled low-cost testing of the potential impact of machine learning tools for solving challenging
problems for MFGs. In \cite{GxHaXr:19}, a general framework of MFGs
is presented for simultaneous learning and
decision making in stochastic games with a large population.
Ruthotto \textit{et al}. \cite{RlOs:20} provide a machine learning framework for solving high-dimensional
MFGs by combining Lagrangian and Eulerian viewpoints and leveraging recent advances.
In \cite{PlGm:20}, the authors provide a rigorous mathematical formulation of deep learning methodologies through an in-depth analysis of the learning procedures characterizing Neural Network models within the theoretical frameworks of stochastic optimal control and MFGs. Luo \textit{et al}. \cite{LjZh:22} apply the deep neural network approach to solving the finite state MFGs. Deep learning applications of MFGs and mean-field control (MFC) in finance
are described in \cite{CrLm:21}, and the reinforcement learning
for MFGs with a special focus on a two-time-scale procedure are considered
in \cite{AaFjLm:21}. Guo \textit{et al}. \cite{GxHaXr:22} extend reinforcement learning work to more general MFGs settings. Two scalable algorithms are used to compute Nash equilibria in various finite-horizon MFGs in \cite{LmPs:22}.

In this paper, we develop a hybrid algorithm to solve
finite-horizon MFGs. The proposed algorithm is based on
MCAM and stochastic approximation (SA) algorithm;
see, e.g., \cite{CxJzYh:20} and \cite{JzYhYg:21} for
the convergence and applications of the hybrid algorithm applied in
an infinite-horizon optimization problem, respectively. Specifically,
we apply MCAM to establish the framework
of the algorithm and identify a neighborhood of optimal controls, which
is critical to fit the initial values of neural network's parameters.
Meanwhile, SA algorithm is used to calibrate
the parameters of the neural network. As for the introduction and
in-depth development of MCAM and stochastic
approximation algorithm, the relevant references can be found in \cite{KhDp:01}
and \cite{KhYg:03}, respectively. Besides, the induced measure
is adopted to obtain the convergence of an iterative mean-field interaction.
Finally, the convergence of the algorithm is obtained and two numerical
examples are provided to illustrate the results.

The main contributions of this paper are as follows. 
A hybrid algorithm is developed
in our work.
Although it bears some 
resemblance to that used in \cite{JzYhYg:21},
in contrast to the infinite horizon optimization problem, 
 our focus here is devoted 
to the finite-horizon optimization problem.
When it comes to the model, the study of MFGs as a Nash equilibrium seeking problem involves a mean-filed interaction, because of which we need to construct iterative procedure in our algorithm. By hybrid we mean the combination of MCAM and SA algorithm. The MCAM is helpful when constructing an iterative algorithm and approximating initial parametrization of neural network from the outset. It can also determine a bounded region within which the SA algorithm is then used to find more accurate parameters. We utilize the induced measure (see, \cite{BeBaCa:18}) in MCAM and Monte-Carlo algorithm in deep learning to construct the iterative procedure of the mean-field interaction.
Dealing with the iteration of the mean-field interaction, we pass the average distributions to the next iteration, where the average distributions means that we assign different weight values to the previous iteration and the current iteration. In addition, our proposed algorithm can enhance the performance of optimization problems in high-dimensional stochastic environments by improving computational efficiency and accuracy. 
Addressing optimization problems involving multiple control variables and states, one inevitably meets the challenge of the ``curse of dimensionality'', wherein the number of computation nodes increases exponentially. The advancements in deep learning allow us to substitute the optimization over piecewise control grids for every state value with the stochastic approximation approach utilizing parametric neural networks across all state values. Now the computational complexity primarily arises from the evaluation of gradients for every state value, turning into a linear relationship with the number of points in the state lattice. In this way, the computation efficiency can be improved. On the other hand, the piecewise constant control is used to approximate the optimal control in traditional methods, and the accuracy of the resulting control strategy is contingent on the density of the control grid. In contrast, the utilization of neural networks in the control strategy enables continuous values for control variables, effectively tackling the challenge of selecting an appropriate precision in control spaces with significant differences in scales. Hence the accuracy of numerical results can be improved as well.

The remainder of this paper is organized as follows. Section \ref{Model} introduces
the framework of MFGs. The main steps of
the hybrid algorithm is presented in Section \ref{Numerical algorithm}. In Section \ref{Convergence of algorithm}, we establish the convergence of the algorithm. Section \ref{Numerical examples} gives two numerical examples for illustration. Finally, Section \ref{Conclusion} concludes the paper with further remarks.

\section{Model}\label{Model}
In this section, we give a brief introduction to the relevant MFG models.
For a fixed finite time horizon $T$, let $\{ \boldsymbol{W}_{t}\}_{0\leq t\leq T}$ be a $d$-dimensional Brownian motion defined on a complete filtered probability space $(\Omega,\mathcal{F},\mathbb{F}=\{ \mathcal{F}_{t}\}_{0\leq t\leq T},\pr)$. Let $\mathcal{Q}\subseteq\mathbb{R}^{d}$ be the space of the state process, $\boldsymbol{U}\subseteq\mathbb{R}^{k}$, and $\mathcal{P}^{p}(\mathcal{Q})$ be the space of probability measures on $\mathcal{Q}$ with finite $p$-th moment, i.e., $m\in\mathcal{P}^{p}(\mathcal{Q})$ if
$\int_{\mathcal{Q}}\left\Vert \boldsymbol{x}\right\Vert ^{p}dm(\boldsymbol{x})<\infty.$
The evolution of state process $\boldsymbol{X}_{s}$ follows
\begin{equation}
\setlength{\abovedisplayskip}{3pt}
\setlength{\belowdisplayskip}{3pt}
\left\{ \begin{array}{l}
d\boldsymbol{X}_s=\boldsymbol{b}(s,\boldsymbol{X}_{s},m_{s},\boldsymbol{\alpha}_{s})ds+\boldsymbol{\varSigma}(s, \boldsymbol{X}_{s}) d\boldsymbol{W}_s, \\
\boldsymbol{X}_t=\boldsymbol{x},
\end{array}\right.\label{eq:2-1-2}
\end{equation}
where $m_{s}$ is the distribution of state $\boldsymbol{X}_{s}$ for given $s\in[t, T]$, and $\boldsymbol{\alpha}_{s}\in\boldsymbol{U}$ is the control. Here, the drift
$\boldsymbol{b}:[t, T] \times \mathcal{Q} \times \mathcal{P}^2\left(\mathcal{Q}\right) \times \boldsymbol{U} \rightarrow \mathbb{R}^d$, and the volatility $\boldsymbol{\varSigma}: [t, T] \times\mathcal{Q} \rightarrow \mathbb{R}^{d\times d}$ are suitable functions.

Denote by $\mathcal{M}\left([t,T];\mathcal{P}^{2}(\mathcal{Q})\right)$ the set of continuous $\mathcal{P}^{2}(\mathcal{Q})$ probability measure-valued processes that are $\mathcal{F}_{s}$-adapted for $s\in [t,T]$, where the topology on $\mathcal{P}^{2}(\mathcal{Q})$ is induced by the Wasserstein metric $\mathcal{W}_{2}$. Let $\mathcal{I}^{2}\left([t,T];\boldsymbol{U}\right)$ be the set of all $\mathcal{F}_{s}$-progressively measurable $\boldsymbol{U}$-valued square-integrable process.
Given 
the probability measures $m=\left(m_{s}\right)_{t\leq s\leq T} \in \mathcal{M}\left([t,T];\mathcal{P}^{2}(\mathcal{Q})\right)$, the value function is defined by
\begin{equation}
\setlength{\abovedisplayskip}{3pt}
\setlength{\belowdisplayskip}{3pt}
\begin{aligned}	
v(t, \boldsymbol{x}) & :=\inf_{\boldsymbol{\alpha}_{s}}\ep_{t,\boldsymbol{x}}\Big[\int_{t}^{T}f(s,\boldsymbol{X}_{s},m_{s},\boldsymbol{\alpha}_{s})ds+g(\boldsymbol{X}_{T},m_{T})\Big],\label{eq:2-1-1}
\end{aligned}
\end{equation}
where $\ep_{t, \boldsymbol{x}}$ denotes the conditional expectation given
$\boldsymbol{X}_{t}=\boldsymbol{x}$, the running cost
$f:[t, T] \times \mathcal{Q} \times \mathcal{P}^2(\mathcal{Q}) \times \boldsymbol{U} \rightarrow$ $\mathbb{R}$, and
the terminal cost
$g:\mathcal{Q} \times \mathcal{P}^2(\mathcal{Q}) \rightarrow$ $\mathbb{R}$ are both measurable functions.

In this paper, we restrict ourselves to equilibrium given by Markovian strategy in feedback (closed-loop) form $\boldsymbol{\alpha}_s=\boldsymbol{\phi}\left(s, \boldsymbol{x}\right)$ with $ \boldsymbol{X}_s=\boldsymbol{x}$, for a deterministic function $\boldsymbol{\phi}$.
Denote by $\mathcal{A}$ the set of admissible controls over the interval $[0, T]$, which is defined as follows.
\begin{defn}
Denote $\mathscr{F}_{s}=\sigma\left\{ \boldsymbol{X}_{t},\boldsymbol{\alpha}_{t},0\leq t\leq s\right\} $.
The control $\boldsymbol{\alpha}_{t}$ is said to
be an admissible feedback control if it satisfies the following conditions:
\begin{itemize}
  \item [1)]
  $\left(\boldsymbol{\alpha}_s\right)_{t \leq s \leq T}\in\mathcal{I}^{2}\left([t,T];\boldsymbol{U}\right)$;
  \item [2)]
  $\pr\left(\boldsymbol{X}_{s}\in A\mid\mathscr{F}_{t}\right)=\pr\left(\boldsymbol{X}_{s}\in A\mid\sigma(\boldsymbol{X}_{t})\right)$, for any $A\in\mathscr{F}_{t}$.
\end{itemize}
\end{defn}
Recall that the Hamiltonian of the stochastic control problem is (the family $m=\left(m_t\right)_{0 \leq t \leq T}$ of probability distributions has been frozen)
\begin{equation*}
\setlength{\abovedisplayskip}{3pt}
\setlength{\belowdisplayskip}{3pt}
	H^{m_t}(t, \boldsymbol{x}, \boldsymbol{y}, \boldsymbol{\alpha})=\boldsymbol{y}^{\top} \boldsymbol{b}\left(t, \boldsymbol{x}, m_t, \boldsymbol{\alpha}\right)+f\left(t, \boldsymbol{x}, m_t, \boldsymbol{\alpha}\right).
\end{equation*}
Then
the value function $v$ is the solution of the HJB equation
\begin{equation}
\setlength{\abovedisplayskip}{3pt}
\setlength{\belowdisplayskip}{3pt}
\left\{ \begin{array}{l}
	\partial_t v+\frac{\boldsymbol{a}(t,\boldsymbol{x})}{2} \partial_{\boldsymbol{x} \boldsymbol{x}}^2 v +\mathcal{H}^{m_t}\left(t, \boldsymbol{x}, \partial_{\boldsymbol{x}} v(t, \boldsymbol{x})\right)=0,\\
v(T, \boldsymbol{x})=g\left(\boldsymbol{x}, m_T\right),
\end{array}\right.\label{eq:2-3-1}
\end{equation}
where $\boldsymbol{a}(t,\boldsymbol{x}):=\boldsymbol{\varSigma}(t,\boldsymbol{x})\boldsymbol{\varSigma}^{\top}(t,\boldsymbol{x})=\{a_{ij}(t,\boldsymbol{x})\},i,j=1,\ldots,d $, and $\mathcal{H}^{m_t}(t, \boldsymbol{x}, \boldsymbol{y}):=\inf _{\boldsymbol{\alpha} \in \mathcal{A}} H^{m_t}(t, \boldsymbol{x}, \boldsymbol{y}, \boldsymbol{\alpha})$.
\begin{defn}
(MFGs equilibrium). The (control, distribution)
$\boldsymbol{\alpha}^*=\left(\boldsymbol{\alpha}_s^*\right)_{t \leq s \leq T} \in \mathcal{I}^2\left([t, T]; \boldsymbol{U}\right)$, $m^* \in \mathcal{M}\left([t, T] ; \mathcal{P}^2\left(\mathcal{Q}\right)\right)$ is a mean-field equilibrium to the MFGs, if $\boldsymbol{\alpha}^*$ solves (\ref{eq:2-1-1}) given the stochastic measure $m^*$, and the distribution of the optimal path $\boldsymbol{X}_s^{\boldsymbol{\alpha}^*}$ coincides with the measure $m_s^*=\mathcal{L}\left(\boldsymbol{X}_s^{\boldsymbol{\alpha}^*}\right)$, where $\mathcal{L}(\cdot)$ is the law.
\end{defn}
In view of \cite{As:15}, we give the following assumptions that play a key role in the subsequent convergence analysis.
\begin{asm}\label{assumption 1}
Let $K$ be the same constant (for notational simplicity) for all assumptions below.

\noindent(1) (Lipschitz) $\partial_{\boldsymbol{x}}f,\partial_{\boldsymbol{\alpha}}f,\partial_{\boldsymbol{x}}g$
exist and are $K$-Lipschitz continuous in $(\boldsymbol{x},\boldsymbol{\alpha},m)$
uniformly in $t$, i.e., for any $t\in[0,T]$, $\boldsymbol{x},\boldsymbol{x}^{\prime}\in\mathcal{Q},\boldsymbol{\alpha},\boldsymbol{\alpha}^{\prime}\in\boldsymbol{U},m,m^{\prime}\in\mathcal{P}^{2}\left(\mathcal{Q}\right)$
\begin{equation*}
\setlength{\abovedisplayskip}{3pt}
\setlength{\belowdisplayskip}{3pt}
\begin{aligned} & \left\Vert \partial_{\boldsymbol{x}}g(\boldsymbol{x},m)-\partial_{\boldsymbol{x}}g(\boldsymbol{x}^{\prime},m^{\prime})\right\Vert \leq K\left(\left\Vert \boldsymbol{x}-\boldsymbol{x}^{\prime}\right\Vert +\mathcal{W}_{2}(m,m^{\prime})\right),\\
 & \left\Vert \partial_{\boldsymbol{x}}f(t,\boldsymbol{x},m,\boldsymbol{\alpha})-\partial_{\boldsymbol{x}}f(t,\boldsymbol{x}^{\prime},m^{\prime},\boldsymbol{\alpha}^{\prime})\right\Vert \\
 & \leq K\left(\left\Vert \boldsymbol{x}-\boldsymbol{x}^{\prime}\right\Vert +\left\Vert \boldsymbol{\alpha}-\boldsymbol{\alpha}^{\prime}\right\Vert +\mathcal{W}_{2}(m,m^{\prime})\right),\\
 & \left\Vert \partial_{\boldsymbol{\alpha}}f(t,\boldsymbol{x},m,\boldsymbol{\alpha})-\partial_{\boldsymbol{\alpha}}f(t,\boldsymbol{x}^{\prime},m^{\prime},\boldsymbol{\alpha}^{\prime})\right\Vert \\
 & \leq K\left(\left\Vert \boldsymbol{x}-\boldsymbol{x}^{\prime}\right\Vert +\left\Vert \boldsymbol{\alpha}-\boldsymbol{\alpha}^{\prime}\right\Vert +\mathcal{W}_{2}(m,m^{\prime})\right),\\
 & \left\Vert \boldsymbol{b}_{0}(t,m)-\boldsymbol{b}_{0}\left(t,m^{\prime}\right)\right\Vert \leq K\mathcal{W}_{2}(m,m^{\prime}),
\end{aligned}
\end{equation*}
where the diffusion coefficient $\boldsymbol{\varSigma}(t,\boldsymbol{x})$ is $K$-Lipschitz in $\boldsymbol{x}$ uniformly in $t$: $\| \boldsymbol{\varSigma}(t, \boldsymbol{x})-\boldsymbol{\varSigma}(t, \boldsymbol{x}^{\prime})\| \leq K\|\boldsymbol{x}-\boldsymbol{x}^{\prime}\|$.

\noindent(2) (Growth) $\partial_{\boldsymbol{x}}f,\partial_{\boldsymbol{\alpha}}f,\partial_{\boldsymbol{x}}g$
satisfy a linear growth condition, i.e., for any $t\in[0,T]$, $\boldsymbol{x}\in\mathcal{Q},\boldsymbol{\alpha}\in\boldsymbol{U},m\in\mathcal{P}^{2}\left(\mathcal{Q}\right)$,
\vspace{-6pt}
\begin{equation*}
\setlength{\abovedisplayskip}{3pt}
\setlength{\belowdisplayskip}{3pt}
\begin{aligned} & \Vert\partial_{\boldsymbol{x}}g(\boldsymbol{x},m)\Vert\leq K[1+\|\boldsymbol{x}\|+(\int_{\mathcal{Q}}\|\boldsymbol{y}\|^{2}\mathrm{~d}m(\boldsymbol{y}))^{\frac{1}{2}}],\\
 & \Vert\partial_{\boldsymbol{x}}f(t,\boldsymbol{x},m,\boldsymbol{\alpha})\Vert\leq K[1+\|\boldsymbol{x}\|+\|\boldsymbol{\alpha}\|+(\int_{\mathcal{Q}}\|\boldsymbol{y}\|^{2}\mathrm{~d}m(\boldsymbol{y}))^{\frac{1}{2}}],\\
 & \Vert\partial_{\boldsymbol{\alpha}}f(t,\boldsymbol{x},m,\boldsymbol{\alpha})\Vert\leq K[1+\|\boldsymbol{x}\|+\|\boldsymbol{\alpha}\|+(\int_{\mathcal{Q}}\|\boldsymbol{y}\|^{2}\mathrm{~d}m(\boldsymbol{y}))^{\frac{1}{2}}].
\end{aligned}
\end{equation*}
In addition $f,g$ satisfy quadratic growth condition in $m$ :
\begin{equation*}
\setlength{\abovedisplayskip}{3pt}
\setlength{\belowdisplayskip}{3pt}
\begin{gathered}\left|g(0,m)\right|\leq K(1+\int_{\mathcal{Q}}\|\boldsymbol{y}\|^{2}\mathrm{~d}m(\boldsymbol{y})),
\end{gathered}
\end{equation*}
\begin{equation*}
\setlength{\abovedisplayskip}{3pt}
\setlength{\belowdisplayskip}{3pt}
\begin{gathered}\left|f(t,0,m,0)\right|\leq K(1+\int_{\mathcal{Q}}\|\boldsymbol{y}\|^{2}\mathrm{~d}m(\boldsymbol{y})).
\end{gathered}
\end{equation*}
(3) (Convexity) $g$ is convex in $\boldsymbol{x}$ and $f$ is convex
jointly in $(\boldsymbol{x},\boldsymbol{\alpha})$ with strongly convex
in $\boldsymbol{\alpha}$, i.e., for any $\boldsymbol{x},\boldsymbol{x}^{\prime}\in\mathcal{Q},\boldsymbol{\alpha},\boldsymbol{\alpha}^{\prime}\in\boldsymbol{U},m\in\mathcal{P}^{2}\left(\mathcal{Q}\right)$,
\begin{equation*}
\setlength{\abovedisplayskip}{3pt}
\setlength{\belowdisplayskip}{3pt}
\left(\partial_{\boldsymbol{x}}g(\boldsymbol{x},m)-\partial_{\boldsymbol{x}}g\left(\boldsymbol{x}^{\prime},m\right)\right)^{T}\left(\boldsymbol{x}-\boldsymbol{x}^{\prime}\right)\geq0,
\end{equation*}
and there exist a constant $c_{f}>0$ such that for any $t\in[0,T],\boldsymbol{x},\boldsymbol{x}^{\prime}\in\mathcal{Q},\boldsymbol{\alpha},\boldsymbol{\alpha}^{\prime}\in\boldsymbol{U},m\in\mathcal{P}^{2}\left(\mathcal{Q}\right)$,
\begin{equation*}
\setlength{\abovedisplayskip}{3pt}
\setlength{\belowdisplayskip}{3pt}
\begin{aligned}
& f(t,\boldsymbol{x}',m,\boldsymbol{\alpha}')-f(t,\boldsymbol{x},m,\boldsymbol{\alpha})-\partial_{\boldsymbol{x}}f(t,\boldsymbol{x},m,\boldsymbol{\alpha})^{\top}(\boldsymbol{x}^{\prime}-\boldsymbol{x})\\
 & \geq\partial_{\boldsymbol{\alpha}}f(t,\boldsymbol{x},m,\boldsymbol{\alpha})^{\top}(\boldsymbol{\alpha}^{\prime}-\boldsymbol{\alpha})+c_{f}\left\Vert \boldsymbol{\alpha}^{\prime}-\boldsymbol{\alpha}\right\Vert ^{2}.
\end{aligned}
\end{equation*}
(4) (Separable in $\boldsymbol{\alpha}$,$m$) $f$ is of the form
\begin{equation*}
\setlength{\abovedisplayskip}{3pt}
\setlength{\belowdisplayskip}{3pt}
f(t,\boldsymbol{x},m,\boldsymbol{\alpha})=f^{0}(t,\boldsymbol{x},\boldsymbol{\alpha})+f^{1}(t,\boldsymbol{x},m),
\end{equation*}
where $f^{0}$ is assumed to be convex in $(\boldsymbol{x},\boldsymbol{\alpha})$
and strongly convex in $\boldsymbol{\alpha}$, and $f^{1}$ is assumed
to be convex in $\boldsymbol{x}$.

\noindent(5) (Weak monotonicity) For all $t\in[0,T],m,m^{\prime}\in\mathcal{P}^{2}\left(\mathcal{Q}\right)$
and $\varXi\in\mathcal{P}^{2}\left(\mathcal{Q}\times\mathcal{Q}\right)$
with marginals $m,m^{\prime}$ respectively,
\begin{equation*}
\setlength{\abovedisplayskip}{3pt}
\setlength{\belowdisplayskip}{3pt}
\begin{aligned}
& \int_{\mathcal{Q}\times\mathcal{Q}}\left[(\partial_{\boldsymbol{x}}g(\boldsymbol{x},m)-\partial_{\boldsymbol{x}}g(\boldsymbol{y},m^{\prime}))^{T}(\boldsymbol{x}-\boldsymbol{y})\right]\varXi(\mathrm{d}\boldsymbol{x},\mathrm{~d}\boldsymbol{y})\geq0,\\
 & \int_{\mathcal{Q}\times\mathcal{Q}}\left[(\partial_{\boldsymbol{x}}f(t,\boldsymbol{x},m,\boldsymbol{\alpha})-\partial_{\boldsymbol{x}}f(t,\boldsymbol{y},m^{\prime},\boldsymbol{\alpha}))^{T}(\boldsymbol{x}-\boldsymbol{y})\right]\\
 & \times\varXi(\mathrm{d}\boldsymbol{x},\mathrm{~d}\boldsymbol{y})\geq0.
\end{aligned}
\end{equation*}
\end{asm}
\begin{remark}
Note that Assumption \ref{assumption 1} extends conditions A and B in \cite{As:15} by considering a general drift coefficient $\boldsymbol{b}(t,\boldsymbol{x},m,\boldsymbol{\alpha})$. Under Assumption \ref{assumption 1} (1) and (3), there is a unique optimal control $\hat{\boldsymbol{\alpha}}$ to solve the MFGs, and the optimal control $\hat{\boldsymbol{\alpha}}$ is Lipschitz continuous in $(\boldsymbol{x}, m)$ uniformly in $t$. Under Assumption \ref{assumption 1} (1)-(5), there is a unique solution to the MFGs. A candidate function $f$, $g$ and $\boldsymbol{b}$ that satisfies Assumption \ref{assumption 1} is given as $f(t, \boldsymbol{x}, m, \boldsymbol{\alpha})=A \boldsymbol{\alpha}^{\top}\boldsymbol{\alpha}+B\left(\boldsymbol{x}-\int \boldsymbol{z} d m(\boldsymbol{z})\right)^2$, $g(\boldsymbol{x}, m)=C\left(\boldsymbol{x}-\int \boldsymbol{z} d m(\boldsymbol{z})\right)^2$ and $\boldsymbol{b}(t,\boldsymbol{x},m,\boldsymbol{\alpha})=\boldsymbol{b}_{0}(t,m)+\boldsymbol{b}_{1}^{\top}(t)\boldsymbol{x}+\boldsymbol{b}_{2}^{\top}(t)\boldsymbol{\alpha}$, where $A,B,C>0$, $\boldsymbol{b}_{0}\in\mathbb{R}^{d},\boldsymbol{b}_{1}\in\mathbb{R}^{d\times d}$
and $\boldsymbol{b}_{2}\in\mathbb{R}^{d\times k}$ are measurable
functions and bounded by $K$.
\end{remark}

\section{Numerical algorithm}\label{Numerical algorithm}

\subsection{MCAM}
In this section, following \cite{KhDp:01}, we
introduce MCAM briefly, and construct the transition probabilities of MCAM for an
iterative computational scheme of the MFGs.
First, we introduce a pair of stepsizes $h=(h_{1},h_{2})$, where $h_{1}>0$ is the discretization 
step size for state variables, and $h_{2}>0$ is 
the stepsize for
that of the time variable.
Suppose that $N_{h_{2}}=T/h_{2}$ is an integer without loss of generality. Define $\boldsymbol{S}^{h_{1}}:=\{ (k^{1}h_{1},\dots,k^{d}h_{1})^{\top}:k^{i}=0,\pm1,\dots,i=1,\dots,d\} $
and let $\{\boldsymbol{X}_{n}^{h, m},n<\infty\}$ be a discrete-time
controlled Markov chain with state space $\boldsymbol{S}^{h_{1}}$
associated with the parameter $h$ and mean-field variable $m$. Let $\boldsymbol{\alpha}^{h, m}=(\boldsymbol{\alpha}_{0}^{h, m},\boldsymbol{\alpha}_{1}^{h, m},\dots,\boldsymbol{\alpha}_{T}^{h, m})$
and $m^{h}=(m_{0}^{h},m_{1}^{h},\ldots,m_{T}^{h})$ denote the sequences of
the control actions and mean-field interactions (population distribution)
at time $0,1,\dots$, respectively.

Denote by $\pr^{h, m}((\boldsymbol{y},\boldsymbol{z})|\boldsymbol{r})$
the probability that $\boldsymbol{X}$ transits from state $\boldsymbol{y}$
at time $nh_{2}$ to state $\boldsymbol{z}$ at time $(n+1)h_{2}$
conditioned on $\boldsymbol{\alpha}_{n}^{h,m}=\boldsymbol{r}$. Let
$\mathcal{A}^{h}$ denote the collection of ordinary controls,
which is determined by a sequence of measurable functions $F_{n}^{h}(\cdot)$
such that
$\boldsymbol{\alpha}_{n}^{h,m}=F_{n}^{h}(\boldsymbol{X}_{k}^{h,m},k\leq n,\boldsymbol{\alpha}_{k}^{h,m},k<n).$
We say that $\boldsymbol{\alpha}_n^{h, m}$ is admissible for the chain if $\boldsymbol{\alpha}_n^{h, m}$ is $\boldsymbol{U}$-valued random variables, respectively, and the Markov property still holds, namely,
\begin{equation*}
\setlength{\abovedisplayskip}{3pt}
\setlength{\belowdisplayskip}{3pt}
	\begin{aligned}
& \pr \left\{\boldsymbol{X}_{n+1}^{h,m}=\boldsymbol{z} \mid \boldsymbol{X}_k^{h,m}, \boldsymbol{\alpha}_k^{h, m}, k \leq n\right\} \\
= & \pr \left\{\boldsymbol{X}_{n+1}^{h,m}=\boldsymbol{z} \mid \boldsymbol{X}_n^{h,m}, \boldsymbol{\alpha}_n^{h, m} \right\}:=\pr^h\left((\boldsymbol{X}_n^{h,m}, \boldsymbol{z}) \mid \boldsymbol{\alpha}_n^{h,m} \right).
\end{aligned}
\end{equation*}
By using the above-mentioned Markov chain, we can approximate the value
function defined in (\ref{eq:2-1-1}) by
\begin{equation}
\setlength{\abovedisplayskip}{3pt}
\setlength{\belowdisplayskip}{3pt}
\begin{aligned}
v^{h}(t,\boldsymbol{x}):= & \inf_{\boldsymbol{\alpha}^{h,m}\in\mathcal{A}^{h}}\ep_{t, \boldsymbol{x}}\Big[\sum_{n=i'}^{N_{h_{2}}-1}f(nh_{2},\boldsymbol{X}_{n}^{h,m},m_{n}^{h},\boldsymbol{\alpha}_{n}^{h,m})h_{2}\\
 & \hspace*{1.6in}
+g(\boldsymbol{X}_{N_{h_{2}}},m_{N_{h_{2}}})\Big],\label{eq:3-1}
\end{aligned}
\end{equation}
where $\boldsymbol{X}_{N_{h_{2}}}$ and $m_{N_{h_{2}}}$ are the terminal value of the discretized state process and mean-field interaction, respectively.

Now, we need to find transition probabilities such that the sequence $\{\boldsymbol{X}_{n}^{h, m},n<\infty\}$ constructed above is locally consistent w.r.t. (\ref{eq:2-1-2}), that is, it satisfies
\begin{equation}
\setlength{\abovedisplayskip}{3pt}
\setlength{\belowdisplayskip}{3pt}
\begin{aligned}
\ep_{\boldsymbol{x}, n}^{h, \boldsymbol{r}} \triangle \boldsymbol{X}_n^{h, m} & =\boldsymbol{b}(t, \boldsymbol{x}, m, \boldsymbol{r}) h_2+o(h_2), \\
{\rm Cov}_{\boldsymbol{x}, n}^{h, \boldsymbol{r}} \triangle \boldsymbol{X}_n^{h, m} & =\boldsymbol{a}(t,\boldsymbol{x}) h_2+o(h_2), \\
\sup _n\left|\triangle \boldsymbol{X}_n^{h, m}\right| & \rightarrow 0, \quad \text { as } h \rightarrow 0, \label{eq:3-2}
\end{aligned}
\end{equation}
where $\bigtriangleup\boldsymbol{X}_{n}^{h,m}:=\boldsymbol{X}_{n+1}^{h,m}-\boldsymbol{X}_{n}^{h,m}$,
$\ep_{\boldsymbol{x},n}^{h,\boldsymbol{r}}$ and ${\rm Cov}_{\boldsymbol{x},n}^{h,\boldsymbol{r}}$
denote the conditional expectation and covariance given by $\{\boldsymbol{X}_{k}^{h,m},\boldsymbol{\alpha}_{k}^{h,m},k\leq n,\boldsymbol{X}_{n}^{h,m}=\boldsymbol{x},\boldsymbol{\alpha}_{n}^{h,m}=\boldsymbol{r}\}$,
respectively. Under Assumption \ref{assumption 1}, there exists at least one admissible control $\hat{\boldsymbol{\alpha}}(\cdot)$ (we omit the superscript $h$, $m$ on $\hat{\boldsymbol{\alpha}}^{h, m}(\cdot)$) so that we can drop inf in (\ref{eq:2-3-1}). Then (\ref{eq:2-3-1}) can be discretized by using finite difference method with stepsizes $h_{1}$ and $h_{2}$.

As a result, the transition probabilities of MCAM for the MFGs are
constructed as follows (see Appendix A for a detailed proof)
\begin{equation}
\setlength{\abovedisplayskip}{3pt}
\setlength{\belowdisplayskip}{3pt}
\begin{aligned} & \pr^{h}(\boldsymbol{x},\boldsymbol{x}\pm h_{1}\boldsymbol{e}_{i}\mid\hat{\boldsymbol{\alpha}})=\frac{(a_{ii}^{2}/2-\sum_{j\neq i}|a_{ij}|/2+b_{i}^{\pm}h_{1})h_{2}}{h_{1}^{2}},\\
 & \pr^{h}(\boldsymbol{x},\boldsymbol{x}+h_{1}\boldsymbol{e}_{i}\pm h_{1}\boldsymbol{e}_{j}\mid\hat{\boldsymbol{\alpha}})=\frac{a_{ij}^{\pm}h_{2}}{2h_{1}^{2}},\\
 & \pr^{h}(\boldsymbol{x},\boldsymbol{x}-h_{1}\boldsymbol{e}_{i}\pm h_{1}\boldsymbol{e}_{j}\mid\hat{\boldsymbol{\alpha}})=\frac{a_{ij}^{\pm}h_{2}}{2h_{1}^{2}},\\
 & \pr^{h}(\boldsymbol{x},\boldsymbol{x}\mid\boldsymbol{\alpha})=1-\sum\pr^{h}(\boldsymbol{x},\boldsymbol{x}\pm h_{1}\boldsymbol{e}_{i}\mid\hat{\boldsymbol{\alpha}})\\
 & -\sum\pr^{h}(\boldsymbol{x},\boldsymbol{x}\pm h_{1}\boldsymbol{e}_{i}\pm h_{1}\boldsymbol{e}_{j}\mid\hat{\boldsymbol{\alpha}}),\label{eq:3-3}
\end{aligned}
\end{equation}
where $b_{i}^{+}$ and $b_{i}^{-}$ 
are
positive and negative parts $b_i$,
respectively. We can also define $a_{ij}^{+}$ and $a_{ij}^{-}$ similar to $b_{i}^{+}$ and $b_{i}^{-}$.

In contrast to a classical optimal control problem, a solution of the MFGs will involve a mean-field interaction $m$. In what follows, we introduce the induced measure (see, \cite{BeBaCa:18}) defined by $\Phi^{h}(m):=\mathbb{P}\circ(\hat{\boldsymbol{X}}^{h, m})^{-1}$ to construct an iterative mean-field interaction. Here, $\{\hat{\boldsymbol{X}}_{t}^{h,m}\} _{t\in[0,T]}$
is a continuous stochastic process which is linear on $[jh_{2},(j+1)h_{2}]$
and equals $\boldsymbol{X}_{n}^{h,m}$ at $n=jh_{2}$, for $j=0,\ldots N_{h_{2}}-1$,
where $\{\boldsymbol{X}_{n}^{h,m}\} $ is the controlled
Markov chain with the optimal feedback control $\hat{\boldsymbol{\alpha}}^{h,m}$.

\begin{construction}\label{construction 1}
Fix $T$ and $(\boldsymbol{x},m^{0})\in\boldsymbol{S}^{h_{1}}\times\mathcal{P}^{2}(\mathcal{Q})$. Let $\{\boldsymbol{X}_{n}^{h, m^{0}}\}$ be the Markov
chain associated with the optimal control $\hat{\boldsymbol{\alpha}}^{h, m^{0}}$.
Having defined the process $\{\boldsymbol{X}_{n}^{h, m^{n}}\} $
for $n\in\mathbb{N}$, set $m^{n+1}:=\Phi^{h}(m^{n})$
and let $\{\boldsymbol{X}_{n}^{h, m^{n+1}}\} $ be the
Markov chain associated with the optimal control $\hat{\boldsymbol{\alpha}}^{h, m^{n+1}}$.
\end{construction}
In general, the accuracy of control strategies obtained by MCAM method is subject to the denseness of the grids, which depends on the types and ranges of controls and states. When the ranges of controls and states are not comparable, finding a suitable stepsize for the lattice becomes challenging, which significantly impacts computational efficiency and accuracy. In what follows, we will present a new deep-learning method to reduce the computation burden in MCAM when the dimension of controls is large. Neural networks are constructed to approximate multiple controls. Let $\theta$ be the collection of all weights and bias terms in the neural networks, and denote by $N(t,\boldsymbol{x},\theta)$ the control strategy.
For simplicity, we assume that the range of the time and state variables are both approximated by $\tilde{n}$ nodes in the setup of neural network. That is, the time and state variables are approximated by $\{t_{\iota}\}_{\iota=1}^{\tilde{n}}$ and $\{\boldsymbol{x}_{\iota}\}_{\iota=1}^{\tilde{n}}$, respectively. Denote by $\theta_{k}^{h}$ and $N(t,\boldsymbol{x},\theta_{k}^{h})$ the $k$-th iterative parameters and control strategy, respectively. Then the formula (\ref{eq:3-1}) can be approximated as
\begin{equation*}
\setlength{\abovedisplayskip}{3pt}
\setlength{\belowdisplayskip}{3pt}
\begin{aligned}
& \hspace*{-0.1 in} v^{h}(t_{\iota},\boldsymbol{x}_{\iota}):=\inf_{\theta_{k}^{h}}J^{h,m}\left(t_{\iota},\boldsymbol{x}_{\iota},N(t_{\iota},\boldsymbol{x}_{\iota},\theta_{k}^{h})\right)\\
 & \ :=  \inf_{\theta_{k}^{h}}\ep_{t_{\iota},\boldsymbol{x}_{\iota}}
\Big[\sum_{n=i'}^{N_{h_{2}}-1}f\left(nh_{2},\boldsymbol{X}_{n}^{h,m},m_{n}^{h},
N(t_{\iota},\boldsymbol{x}_{\iota},\theta_{k}^{h})\right)h_{2}
\\
 & \hspace*{2.2in}
 +g(\boldsymbol{X}_{N_{h_{2}}}^{h,m},m_{N_{h_{2}}}^{h})\Big].
\end{aligned}
\end{equation*}
Now, assuming we are currently in the $k$-th iteration with the iterative
value function $v_{k-1}^{h}$ obtained from the previous iteration,
then the system of dynamic programming equations in the $k$-th iteration
follows
\begin{equation}
\setlength{\abovedisplayskip}{3pt}
\setlength{\belowdisplayskip}{3pt}
v_{k}^{h}(nh_{2},\boldsymbol{x})=S(nh_{2}
+h_{2},\boldsymbol{x},v_{k-1}^{h}(nh_{2}+h_{2},\boldsymbol{x}),\hat{\boldsymbol{\alpha}}),\label{eq:4-1}
\end{equation}
where
\begin{equation*}
\setlength{\abovedisplayskip}{3pt}
\setlength{\belowdisplayskip}{3pt}
\begin{aligned} & S(nh_{2}+h_{2},\boldsymbol{x},v_{k-1}^{h}(nh_{2}+h_{2},\boldsymbol{x}),\hat{\boldsymbol{\alpha}})\\
 & :=v_{k-1}^{h}(nh_{2}+h_{2},\boldsymbol{x})\pr^{h}(\boldsymbol{x},\boldsymbol{x}\mid\hat{\boldsymbol{\alpha}})\\
 & +\sum_{i=1}^{d}v_{k-1}^{h}(nh_{2}+h_{2},\boldsymbol{x}\pm h_{1}\boldsymbol{e}_{i})\pr^{h}(\boldsymbol{x},\boldsymbol{x}\pm h_{1}\boldsymbol{e}_{i}\mid\hat{\boldsymbol{\alpha}})\\
 & +\sum_{i=1}^{d}\sum_{j\neq i}v_{k-1}^{h}(nh_{2}+h_{2},\boldsymbol{x}\pm h_{1}\boldsymbol{e}_{i}\pm h_{1}\boldsymbol{e}_{j})\\
 & \quad\times\pr^{h}(\boldsymbol{x},\boldsymbol{x}\pm h_{1}\boldsymbol{e}_{i}\pm h_{1}\boldsymbol{e}_{j}\mid\hat{\boldsymbol{\alpha}})\\
 & +f(nh_{2}+h_{2},\boldsymbol{x},m,\hat{\boldsymbol{\alpha}})h_{2}\Big].
\end{aligned}
\end{equation*}
Then the equation (\ref{eq:4-1})
can be rewritten as $v_{k}^{h}(t_{\iota},\boldsymbol{x}_{\iota})=S\left(t_{\iota+1},\boldsymbol{x}_{\iota},v_{k-1}^{h}(t_{\iota+1},\boldsymbol{x}_{\iota}),N(t_{\iota},\boldsymbol{x}_{\iota},\theta_{k}^{h})\right), 1\leq\iota\leq n.$
We aim to show
\begin{equation}
\setlength{\abovedisplayskip}{3pt}
\setlength{\belowdisplayskip}{3pt}
\lim_{h\rightarrow0,k\rightarrow\infty}\theta_{k}^{h}=\theta^{*}, \label{eq:4-2}
\end{equation}
where $\theta^{*}$ is the local optimum.
\begin{remark}
Note that
 $h$ is the control grid and that $k$ is the iteration in the objective (\ref{eq:4-2}). 
 The iteration $k$ based on gradient descent method converges faster than the control grid $h$ based on the grid search method.
\end{remark}

\subsection{SA algorithm}
In this section, an
SA algorithm is introduced to obtain
$\lim_{l\rightarrow\infty}\theta_{k,l}^{h}=\theta_{k}^{h},$
where $\theta_{k,l}^{h}$ is the $l$-th estimate of the optimum of
$\theta$ in the $k$-th iteration. Given the time lattice $\{t_{\iota}\}_{\iota=1}^{n}$ and state lattice $\{\boldsymbol{x}_{\iota}\}_{\iota=1}^{n}$, define the global
improvement function $G^{h}(\theta^{h})$ as
\begin{equation}
\setlength{\abovedisplayskip}{3pt}
\setlength{\belowdisplayskip}{3pt}
G^{h}(\theta^{h}):= G^{h}\left(J^{h,m}\left(t_{\iota},\boldsymbol{x}_{\iota},N(t_{\iota},\boldsymbol{x}_{\iota},\theta^{h})\right),\; \iota=1,\ldots,n\right). \label{eq:4-3}
\end{equation}
The main goal is to use $N(t_{\iota},\boldsymbol{x}_{\iota},\theta^{h})$
and choose $\theta^{h}$ to maximize the opposite of the global improvement function
$G^{h}(\theta^{h})$. The choice of $G^{h}$ should serve
the goal that the value function will be improved on most states rather
than on every state of the time and state lattices. Usually, $G^{h}$ can be
chosen as the average value of the value function for simplicity.
We can also choose other general global improvement functions such
as the weighted average of the value function. The impact to efficiency
depends on the formulation of the finite-horizon optimization problem.

To proceed, a general setting of SA algorithm is provided for the
proof of convergence in Section \ref{Convergence of SA}. Without loss of generality, we assume that the parameters of neural network $\theta$ is an $r$-dimension
vector. Let $\theta_{l}$ be the $l$-th estimate of the optimum,
and denote by $\delta_{l}$ and $\varepsilon_{l}$ the stepsizes of
finite difference intervals and iterations, respectively. To simplify
the notation, we omit the subscript $k$ and the superscript $h$
in each term, and we write the global improvement function $G^{h}(\theta^{h})$
as $G(\theta)$. For more details of the procedure of SA algorithm, we refer the reader to \cite{KhYg:03}.
%
To verify the convergence of SA algorithm, the stepsizes $\delta_{l}$
and $\varepsilon_{l}$ satisfy the conditions
$\varepsilon_{l}\rightarrow0$,
$\sum_{l}\varepsilon_{l}=\infty$,
$\varepsilon_{l}/\delta_{l}\rightarrow0$,
$\sum_{l}\varepsilon_{l}^{2}/\delta_{l}^{2}<\infty$.

\begin{remark}
The first two conditions can be found in the pp.120 in \cite{KhYg:03}.
The third condition guarantees that the stepsize goes to $0$ much
faster than the finite difference intervals do, which is the main
requirement in this article.
\end{remark}

Here, we are concerned with the projection
algorithm, where the iterative
$\theta_{l}$ is confined to some bounded set, which
is a common
practice in applications. Define the projection region by
\begin{equation*}
\setlength{\abovedisplayskip}{3pt}
\setlength{\belowdisplayskip}{3pt}
\begin{aligned}
H & =\{\theta:N(t,\boldsymbol{x},\theta)\in[N(t,\boldsymbol{x},\theta_{0})-\delta_{l},N(t,\boldsymbol{x},\theta_{0})+\delta_{l}]\cap \mathcal{A},\\
 & \quad|\theta_{l,j}|\leq M,M\in\mathbb{R},j=1,\ldots,r\} ,
\end{aligned}
\end{equation*}
where $M$ is a sufficiently large positive number, $\theta_{0}$
is the MCAM's piecewise optimal control, and $\theta_{l,j}$ is the
$j$-th element of vector $\theta_{l}$. Then the algorithm is
\begin{equation}
\setlength{\abovedisplayskip}{3pt}
\setlength{\belowdisplayskip}{3pt}
\theta_{l+1}=\Pi_{H}[\theta_{l}+\varepsilon_{l}K_{l}],\label{eq:4-2-1}
\end{equation}
where $\Pi_{H}$ denotes the projection operator onto the constraint
set $H$, and $K_{l}$ is the gradient of $G(\theta)$. Let $\Pi_{H}(\theta)$ denote the closest point in $H$
to $\theta$. If the iteration is within the region we keep their
values; otherwise, we push them back to the boundaries.

\subsection{Summary of algorithm}
With the implementation details explained above, the pseudo-code of the proposed hybrid deep learning algorithm is summarized in Algorithm \ref{algorithm1}. Besides, a pictorial diagram of this algorithm with various modules is shown in Figure \ref{fig:my_label}.
\begin{figure}[http]
  \centering
  \includegraphics[width=0.45\textwidth]{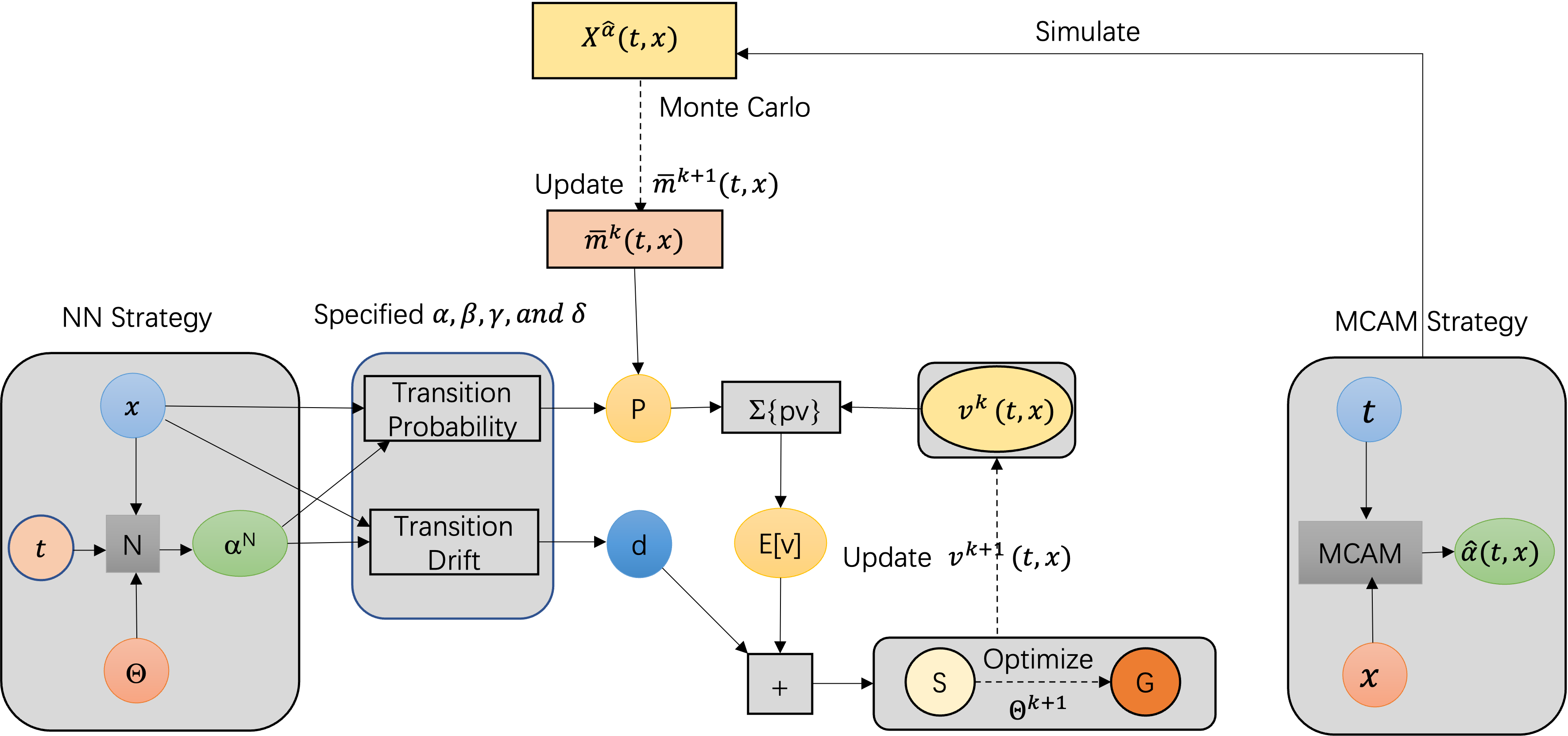}
  \caption{A pictorial diagram of algorithm.}
  \label{fig:my_label}
\end{figure}
\begin{algorithm}[http]
\caption{Framework of the hybrid deep learning method.}
\label{algorithm1}
\begin{algorithmic}[1]
\REQUIRE ~~\\ 
    Time and state lattices $\{t_{\iota}\}_{\iota=1}^{\tilde{n}},\{\boldsymbol{x}_{\iota}\}_{\iota=1}^{\tilde{n}}$ in deep learning;\\
    Time and state lattices $\{t_{\iota'}\}_{\iota'=1}^{n'},\{\boldsymbol{y}_{\iota'}\}_{\iota'=1}^{n'}$ in MCAM;\\
    Initial value function $v_{0}(t_{\iota},\boldsymbol{x}_{\iota})$ in deep learning;\\
    Initial value function $U_{0}(t_{\iota'},\boldsymbol{y}_{\iota'})$ in MCAM;\\
    Set of computation precision, $\tilde{\epsilon}$;\\
    Maximal number of learning times, $\widetilde{N}$;\\
    Initial distribution, $\bar{m}^{(0)}$;
\ENSURE ~~\\ 
    Approximation of the optimal control $N(t_{\iota},\boldsymbol{x}_{\iota},\theta_{k})$;
    \WHILE{$k<\widetilde{N}$}
    \STATE Obtain optimal control $\hat{\boldsymbol{\alpha}}_{k}(t_{\iota'},\boldsymbol{y}_{\iota'})$ by MCAM;
    \STATE Simulate $\boldsymbol{X}^{\hat{\boldsymbol{\alpha}}_{k}(t_{\iota'},\boldsymbol{y}_{\iota'})}$ and obtain $m^{(k)}$ by Monte Carlo;
    \STATE Update the value of $\bar{m}^{(k)}$ by
	$\bar{m}^{(k)}=\frac{k-1}{k} \bar{m}^{(k-1)}+\frac{1}{k} m^{(k)}$;
    \STATE Obtain initial value of the parameter $\theta_{k,0}$ by
	\begin{equation*}
	\setlength{\abovedisplayskip}{3pt}
	\setlength{\belowdisplayskip}{3pt}
	\theta_{k,0}=\operatorname{argmin}_{\theta\in H}\sum_{\iota'=1}^{n'}\left(\Vert \hat{\boldsymbol{\alpha}}_{k}(t_{\iota'},\boldsymbol{y}_{\iota'})-N(t_{\iota'},\boldsymbol{y}_{\iota'},\theta)\Vert\right)^{2};
	\end{equation*}
    \STATE Obtain the control using $G(\theta_{k})$ in (\ref{eq:4-3}) by stochastic approximation algorithm, $\theta_{k,l}=\Pi_{H}[\theta_{k-1,l}+\varepsilon_{l}K_{l}]$;
    \STATE Iterate the value function $U_{k}(t_{\iota'},\boldsymbol{y}_{\iota'})$,
    $U_{k}(t_{\iota'},\boldsymbol{y}_{\iota'})=S\left(t_{\iota'+1},\boldsymbol{y}_{\iota'},U_{k-1}(t_{\iota'+1},\boldsymbol{y}_{\iota'}),N(t_{\iota'},\boldsymbol{y}_{\iota'},\theta_{k})\right)$;
    \STATE Iterate the value function $v_{k}(t_{\iota},\boldsymbol{x}_{\iota})$, $v_{k}(t_{\iota},\boldsymbol{x}_{\iota})=S\left(t_{\iota+1},\boldsymbol{x}_{\iota},v_{k-1}(t_{\iota},\boldsymbol{x}_{\iota}),N(t_{\iota},\boldsymbol{x}_{\iota},\theta_{k})\right)$;
        \IF{$\sum_{\iota=1}^{n}\left(v_{k}(t_{\iota},\boldsymbol{x}_{\iota})-v_{k-1}(t_{\iota},\boldsymbol{x}_{\iota})\right)^{2}<\tilde{\epsilon}$}
            \STATE Stop;
        \ELSE
            \STATE Go to Step 1;
        \ENDIF
    \ENDWHILE
\RETURN $N(t_{\iota},\boldsymbol{x}_{\iota},\theta_{k})$. 
\end{algorithmic}
\end{algorithm}
\begin{remark}
The input values $U_{0}(t_{\iota'},\boldsymbol{y}_{\iota'})$ and $v_{0}(t_{\iota},\boldsymbol{x}_{\iota})$ are both initiated by the formulation of the terminal value function $g$. We use different time and state lattices denoted as $n'$ in MCAM and $n$ in deep learning. 
Specifically, we explain the following points: 1). To ensure consistency in the bounds of their value ranges, we require that these two lattices satisfy the conditions: $\boldsymbol{x}_0 = \boldsymbol{y}_0$ and $\boldsymbol{x}_n = \boldsymbol{y}_{n^{\prime}}$; 2). As is well known, MCAM faces exponential growth in computational complexity with the increase in the number of discrete points, resulting in the ``curse of dimensionality''. Therefore, we use a relatively small number of discrete points $n^{\prime}$ to obtain a rough initial value, which is then refined using the SGD algorithm to fit the parameters for the neural network, providing the initial parameter values for the SA algorithm in deep learning. This approach not only improves computational efficiency but also offers an initial general scale for further processing in deep learning; 3). In deep learning, we use a larger number of discrete points $n$ to train the neural network parameters and achieve more precise computational outcomes. In conclusion, to obtain the optimal parameters of the neural network, we apply MCAM with a coarse scale to estimate the initial guess for the neural network, and apply SA with a fine scale to refine the parameters within a bounded region.
 At step 1, we use Construction \ref{construction 1} to obtain an iterative mean-field interaction $m^{(k)}$, which plays a key role in getting the optimal control that depends on the mean-field interaction. In addition, we adopt the Monte Carlo method to obtain the $k$-th mean-field interaction $\bar{m}^{(k)}$. Specifically, we use a technique of average distributions as follows: (a) At the $k$-th iteration, given the measures $\bar{m}^{(k-1)} \in \mathcal{P}^2\left(\mathbb{R}^d\right)$ in (\ref{eq:2-1-2})-(\ref{eq:2-1-1}), and solve the optimal control in (\ref{eq:2-1-1}) denoted by $\boldsymbol{\alpha}^{k}$; (b) Solve
equation (\ref{eq:2-1-2}) for $\boldsymbol{X}^{\boldsymbol{\alpha}^{(k)}}$ and then infer the distribution $m^{(k)}=\mathcal{L}(\boldsymbol{X}^{\boldsymbol{\alpha}^{(k)}})$; (c) Pass the average distributions $\bar{m}^{(k)}$ to the next iteration.
\end{remark}

\section{Convergence of algorithm}\label{Convergence of algorithm}
This section focuses on the convergence of the algorithm.
We start with an initial guess $\theta_{k,0}^{h}$, then the iteration
will lead to the optimal set of parameters $\theta^{\ast}$.
\subsection{Convergence of SA algorithm}\label{Convergence of SA}
We aim to show
$	\lim_{l\rightarrow\infty}\theta_{k, l}^{h} =\theta_{k}^{h}.$
Define the projection term $z_{l}$ by writing (\ref{eq:4-2-1}) as
\begin{equation}
\setlength{\abovedisplayskip}{3pt}
\setlength{\belowdisplayskip}{3pt}
\theta_{l+1}=\theta_{l}+\varepsilon_{l}K_{l}+\varepsilon_{l}z_{l}.\label{eq:5-1-1}
\end{equation}
Thus $\varepsilon_{l}z_{l}=\theta_{l+1}-\theta_{l}-\varepsilon_{l}K_{l}$;
it is the vector of the shortest Euclidean length needed to take $\theta_{l}+\varepsilon_{l}K_{l}$
back to the constraint set $H$ if it escapes from $H$. Denote by $\widehat{G}(\theta,\eta)$ the observations of $G(\theta)$ with noise $\eta$, and we write $\widehat{G}(\theta,\eta)$ as $\widehat{G}(\theta,\eta)=G_{0}(\theta,\widetilde{\eta})+\widehat{\eta}$,
where $G_{0}(\cdot,\widetilde{\eta})$ is three-times continuously
differentiable for each $\widetilde{\eta}$, and $\widetilde{\eta}$
is the nonadditive noise, which is independent of the additive noise
$\widehat{\eta}$.

Let $\bar{G}(\theta)=\ep_{l}G_{0}(\theta,\widetilde{\eta}_{l}^{\pm})$,
where $\ep_{l}$ is the conditional expectation with respect to the
$\sigma$-field generated by $\{\widetilde{\eta}_{l}\} $.
Next we define the finite difference bias $\{\bar{\omega}_{l}\}$ as
$\bar{\omega}_{l,j}=\frac{\partial\bar{G}(\theta_{l})}{\theta_{l,j}}-\frac{\bar{G}(\theta_{l}+\delta_{l}\boldsymbol{e}_{j})-\bar{G}(\theta_{l}-\delta_{l}\boldsymbol{e}_{j})}{2\delta_{l}}.$
In what follows, we consider the two types of noise, uncorrelated
noise $\{\zeta_{l}\}$ and correlated noise $\{\psi_{l}\}$.
\begin{equation*}
\setlength{\abovedisplayskip}{3pt}
\setlength{\belowdisplayskip}{3pt}
\begin{aligned}
\zeta_{l,j}= & [\bar{G}(\theta_{l}+\delta_{l}\boldsymbol{e}_{j})-G_{0}(\theta_{l}+\delta_{l}\boldsymbol{e}_{j},\widetilde{\eta}_{l,j}^{+})]\\
 & -[\bar{G}(\theta_{l}-\delta_{l}\boldsymbol{e}_{j})-G_{0}(\theta_{l}-\delta_{l}\boldsymbol{e}_{j},\widetilde{\eta}_{l,j}^{-})],\\
\psi_{l,j}= & \widehat{\eta}_{l,j}^{+}-\widehat{\eta}_{l,j}^{-}.
\end{aligned}
\end{equation*}
Here, we use the following notation
\begin{equation*}
\setlength{\abovedisplayskip}{3pt}
\setlength{\belowdisplayskip}{3pt}
\begin{aligned}
 & \zeta_{l}=\left(\zeta_{l,1},\ldots,\zeta_{l,r}\right)^{\prime},\\
 & \psi_{l}=\left(\psi_{l,1},\psi_{l,2},\ldots,\psi_{l,r}\right)^{\prime},\bar{\omega}_{l}=(\bar{\omega}_{l,1},\ldots,\bar{\omega}_{l,r})^{\prime},\\
 & \bar{G}_{l}(\cdot)=\left(\bar{G}_{l,1,\theta_{l,1}}(\cdot),\bar{G}_{l,2,\theta_{l,2}}(\cdot),\ldots,\bar{G}_{l,r,\theta_{l,r}}(\cdot)\right)^{\prime}.
\end{aligned}
\end{equation*}
Then 
algorithm (\ref{eq:5-1-1}) can be written as
\begin{equation}
\setlength{\abovedisplayskip}{3pt}
\setlength{\belowdisplayskip}{3pt}
\theta_{l+1}=\theta_{l}+\varepsilon_{l}\bar{G}_{l,\theta}(\theta_{l})+\varepsilon_{l}\bar{\omega}_{l}+\varepsilon_{l}z_{l}+\varepsilon_{l}\frac{\zeta_{l}}{2\delta_{l}}+\varepsilon_{l}\frac{\psi_{l}}{2\delta_{l}}.\label{eq:5-1-2}
\end{equation}
In fact, the above algorithm is of the KW (Kiefer and Wolfowitz) type, and the last two terms in (\ref{eq:5-1-2}) can be considered as noise terms, in which $\psi_{l}$ and $\zeta_{l}$ represent additive and non-additive noise, respectively.

To establish convergence of the algorithm, we state some sufficient conditions needed for our recursive algorithm. 
\begin{asm}\label{assumption 2}
We make assumptions (1) and (2) as follows.

\noindent(1) The function $\widehat{G}(\cdot,\eta)$ is three-times continuously
differentiable for each $\eta$.

\noindent(2) $\{\widehat{\eta}_{l}\} $ is the stationary martingale
difference sequence satisfying $\ep|\widehat{\eta}_{l}|^{2}<\infty$;
$\{\widetilde{\eta}_{l}\}$ is a sequence of bounded
stationary $\phi$-mixing process with mixing rate $\phi(k)$ such
that $\sum_{k}\phi^{1/2}(k)<\infty$.
\end{asm}
To proceed, define $t_{0}=0$ and $t_{l}=\sum_{j=0}^{l-1}\varepsilon_{j}$.
For $t\geq0$, let $\varDelta(t)$ denote the unique value of $l$ such that
$t_{l}\leq t<t_{l+1}$. For $t<0$, set $\varDelta(t)=0$. Define the continuous-time
interpolation $\theta_{0}(\cdot)$ on $(-\infty,\infty)$ by $\theta_{0}(t)=\theta_{0}$,
for $t\leq0$, and for $t\geq0$, $\theta_{0}(t)=\theta_{l}$, for
$t_{l}\leq t<t_{l+1}$. Define the sequence of shifted process
$\theta_{l}(\cdot)$ as $\theta_{l}(t)=\theta_{0}(t_{l}+t), -\infty<t<\infty$.
Define $Z_{0}(t)=0$, for $t\leq0$ and $Z_{0}(t)=\sum_{i=0}^{\varDelta(t)-1}\varepsilon_{i}z_{i}$, for $t\geq0$, where $z_{i}=0$, for $i<0$.
Define the shifted process $Z_{l}(t)$ by
$Z_l(t)=Z_0(t_l+t)-Z_0(t_l)=\sum_{i=l}^{\varDelta(t_l+t)-1} \varepsilon_i z_i, t \geq 0, \quad Z_l(t)=-\sum_{i=\varDelta(t_l+t)}^{l-1} \varepsilon_i z_i, t<0.$
Then the interpolated process $\theta_{l}(\cdot)$ can be rewritten
as below
\begin{equation*}
\setlength{\abovedisplayskip}{3pt}
\setlength{\belowdisplayskip}{3pt}
\begin{aligned}
\theta_{l}(t) & =\theta_{l}+\sum_{i=l}^{m(t_{l}+t)-1}\varepsilon_{i}\bar{G}_{i,\theta}(\theta_{i})+\sum_{i=l}^{m(t_{l}+t)-1}\varepsilon_{i}\frac{\zeta_{i}}{2\delta_{i}}\\
& \quad+\sum_{i=l}^{m(t_{l}+t)-1}\varepsilon_{i}\frac{\psi_{i}}{2\delta_{i}}+\sum_{i=l}^{m(t_{l}+t)-1}\varepsilon_{i}b_{i}+\sum_{i=l}^{m(t_{l}+t)-1}\varepsilon_{i}z_{i}\\
 & =\theta_{l}+g_{l}(t)+\zeta_{l}(t)+\psi_{l}(t)+b_{l}(t)+Z_{l}(t),
\end{aligned}
\end{equation*}
where
\begin{equation*}
\setlength{\abovedisplayskip}{3pt}
\setlength{\belowdisplayskip}{3pt}
\begin{aligned}
g_{l}(t) & =\sum_{i=l}^{\varDelta(t_{l}+t)-1}\varepsilon_{i}\bar{G}_{i,\theta}(\theta_{i}),\;\zeta_{l}(t)=\sum_{i=l}^{\varDelta(t_{l}+t)-1}\varepsilon_{i}\frac{\zeta_{i}}{2\delta_{i}},\\
\psi_{l}(t) & =\sum_{i=l}^{\varDelta(t_{l}+t)-1}\varepsilon_{i}\frac{\psi_{i}}{2\delta_{i}},\; b_{l}(t)=\sum_{i=l}^{\varDelta(t_{l}+t)-1}\varepsilon_{i}b_{i},\;\text{for \ensuremath{t\geq0}.}
\end{aligned}
\end{equation*}
Next, we define the projected ODE (ordinary differential equation) as follows,
\begin{equation}
\setlength{\abovedisplayskip}{3pt}
\setlength{\belowdisplayskip}{3pt}
\begin{aligned}
& \dot{\theta}(t)=-\bar{G}_\theta(\theta(t))+z(t), \quad z \in C(\theta), \\
& Z(t)=\int_0^t z(u) du.
\end{aligned}\label{eq:5-1-3}
\end{equation}
Here, for $\theta\in M^{0}$, the interior of $M,C(\theta)$
contains the zero element only; and for $\theta\in\partial M$, the
boundary of $M$. Let $C(\theta)$ be the infinite convex cone generated
by the outer normals at $\theta$ of the faces on which $\theta$
lies, and $z(\cdot)$ be the projection or the constraint term that is
the minimum force needed to keep $\theta(\cdot)$ in $M$.
\begin{thm}
Suppose the conditions (1) and (2) in Assumption \ref{assumption 2} are satisfied.
In addition, we have
\begin{equation*}
\setlength{\abovedisplayskip}{3pt}
\setlength{\belowdisplayskip}{3pt}
\sup_{l\leq k\leq m(t_{l}+T)}(\varepsilon_{k}/\delta_{k}^{2})/(\varepsilon_{l}/\delta_{l}^{2})\leq c_{1}(T), \text{ for some }c_{1}(T)<\infty.
\end{equation*}
Then, there is a null set $N$ such that for all $\omega\notin N,\{\theta_{l}(\cdot),Z_{l}(\cdot)\} $
is equicontinuous in the extended sense. Let $(\theta(\cdot),Z(\cdot))$
denote the limit of a convergent subsequence. Then, it satisfies the
projected ODE (\ref{eq:5-1-3}). If $\theta^{*}$ is an asymptotically
stable point of (\ref{eq:5-1-3}) and $\theta_{l}$ is in some compact
set that is a subset of the domain of attraction of $\theta^{*}$
$\text{w.p.}1$, then $\theta_{l}\rightarrow\theta^{*}$ $\text{w.p.}1$.
\end{thm}

\begin{proof}
To prove this theorem, we apply Theorems 4.2.2 and 6.5.1 of \cite{KhYg:03}.
For more details of the proof procedure, we refer the reader to Theorem
5.8 of \cite{JzYhYg:21}.	
\end{proof}

\subsection{Convergence of MCAM}
This section deals with the convergence of MCAM, that is, 
$\lim_{h\rightarrow0,k\rightarrow\infty}\theta_{k}^{h} =\theta^{*}.	$
In what follows, we use the relaxed control (see, \cite{KhDp:01}) to solve the issue of closure with an ordinary control. The use of relaxed control gives us an alternative to obtain and characterize
the weak limit.
\begin{defn}
Let $\mathscr{B}(\boldsymbol{U})$ and $\mathscr{B}(\boldsymbol{U}\times[t,T])$
denote the Borel $\sigma$-algebras of $\boldsymbol{U}$ and $\boldsymbol{U}\times[t,T]$.
An admissible relaxed control or simply a relaxed control $\gamma$
is a measure on $\mathscr{B}(\boldsymbol{U}\times[t,T])$ such that
$\gamma(\boldsymbol{U}\times[t,s])=s-t$ for all $s\in[t,T]$.
\end{defn}

Given the relaxed control $\gamma$, which is progressively measurable with respect to the filtration of Brownian motions, the marginal of $\gamma$ on $[t, s]$ is a Lebesgue measure on $\mathscr{B}(\boldsymbol{U})$, such that $\gamma(d\boldsymbol{r},ds)=\gamma_{s}(d\boldsymbol{r})ds$.
We further define the relaxed control representation $\gamma$
of $\boldsymbol{\alpha}$ by
$\gamma_{t}(B)=I_{\{\boldsymbol{\alpha}_{t}\in B\}}$, for any $B\in\mathscr{B}(\boldsymbol{U}).$
Therefore, we can represent any ordinary admissible control $\boldsymbol{\alpha}$
as a relaxed control by using the definition $\gamma_{t}(\mathrm{d}\boldsymbol{r})=I_{\boldsymbol{\alpha}}(\boldsymbol{r})\mathrm{d}\boldsymbol{r}$,
where $I_{\boldsymbol{\alpha}}(\boldsymbol{r})$ is the indicator
function concentrated at the point $\boldsymbol{\alpha}=\boldsymbol{r}$.
Thus, the measure-valued derivative $\gamma_{t}(\cdot)$ of the relaxed
control representation of $\boldsymbol{\alpha}_{t}$ is a measure that
is concentrated at the point $\boldsymbol{\alpha}_{t}$. For each $t$,
$\gamma_{t}(\cdot)$ is a measure on $\mathscr{B}(\boldsymbol{U})$
satisfying $\gamma_{t}(\boldsymbol{U})=1$ and $\gamma(A)=\int_{\boldsymbol{U}\times[t,T]}I_{\{(\boldsymbol{r},s)\in A\}}\gamma_{t}(d\boldsymbol{r})ds$
for all $A\in\mathscr{B}(\boldsymbol{U}\times[t,T])$.

With the relaxed control representation, we need to further discretize the continuous processes $(\boldsymbol{X}(t),\gamma(t))$ to analyze their convergence. Meanwhile, because of the involvement of $m$, we introduce a small time interval (see, \cite{BeBaCa:18}) upon which the convergence of laws obtained from the iteration scheme can be demonstrated.
Let $q\in(0, 1)$, and we set
$k_{h}:=\min\left\{ k\in\mathbb{N}:\mathcal{W}_{2}^{2}\left(\Phi^{h}(m^{k}),m^{k}\right)\leq\frac{2q}{1-q}h_{1}^{2}\right\}$
and $m_{h}:=m^{k_{h}}$. For the existence of $k_{h}$, we refer the reader to \cite{BeBaCa:18}.
Based on the definition of $m_{h}$, we define the piecewise constant interpolations by
\begin{equation*}
\setlength{\abovedisplayskip}{3pt}
\setlength{\belowdisplayskip}{3pt}	
\begin{aligned}\boldsymbol{X}_{t}^{h,m_{h}} & =\boldsymbol{X}_{n}^{h,m_{h}},\quad\boldsymbol{\alpha}_{t}^{h,m_{h}}=\boldsymbol{\alpha}_{n}^{h,m_{h}},\\
\boldsymbol{W}_{t}^{h,m_{h}} & =\sum_{k=0}^{n-1}(\Delta\boldsymbol{X}_{k}^{h,m_{h}}-\ep_{k}^{h}\Delta\boldsymbol{X}_{k}^{h,m_{h}})/\boldsymbol{\varSigma}(kh_2,\boldsymbol{X}_{k}^{h,m_{h}}),\\
\gamma_{t}^{h,m_{h}}(\cdot) & =\gamma_{n}^{h,m_{h}}(\cdot),\quad\text{for \ensuremath{t\in[nh_{2},(n+1)h_{2}).}}
\end{aligned}
\end{equation*}
where $\ep_{k}^{h}$ denotes the expectation conditioned on the smallest $\sigma$-algebra of $\{\boldsymbol{X}_{k}^{h.m_h}, \boldsymbol{\alpha}_{k}^{h,m_h},k\leq n\}$.

Denote by $\Gamma^{h}$ the set of admissible relaxed control $\gamma^{h,m_{h}}$
such that $\gamma_{s}^{h,m_{h}}$ is a fixed probability measure in
the interval $[nh_{2},(n+1)h_{2})$. Then we rewrite the value function as
\begin{equation}
\setlength{\abovedisplayskip}{3pt}
\setlength{\belowdisplayskip}{3pt}
\begin{aligned}
v^{h,m_{h}}(t,\boldsymbol{x}) =& \inf_{\gamma^{h,m_{h}}\in\Gamma^{h}}\ep_{t,\boldsymbol{x}}\Big[\sum_{n=i'}^{N_{h_{2}}-1}f(nh_{2},\boldsymbol{X}_{n}^{h,m_{h}},m_{n}^{h},\gamma_{n}^{h,m_{h}})
\\
& \hspace*{1.2in}
\times h_{2}+g(\boldsymbol{X}_{N_{h_{2}}},m_{N_{h_{2}}})\Big].
\end{aligned}
\label{eq:5-2-1-2}
\end{equation}
Note that relaxed controls are a device mainly used for mathematical analysis purposes. Nevertheless, they can always be approximated by ordinary controls. This is referred to as a chattering lemma. Here we state the result below and postpone its proof to Appendix B.
\begin{lemma} \label{lemma1}
Fix $\tilde{\eta}>0$, and $\left(\boldsymbol{X}(\cdot),\gamma(\cdot),\boldsymbol{W}(\cdot)\right)$
be an $\tilde{\eta}$-optimal control ($\left(\boldsymbol{X}(\cdot), \gamma(\cdot), \boldsymbol{W}(\cdot)\right)$ being $\tilde{\eta}-$optimal means that these processes constitute a solution to the control problem that is nearly optimal, with a solution no lager than $\tilde{\eta}$). For each $\eta>0$, there is
an $\varepsilon>0$ and a probability space on which are defined $\boldsymbol{W}^{\eta}(\cdot)$,
a control $\boldsymbol{\alpha}^{\eta}(\cdot)$ which is an admissible
finite set $\mathcal{A}_{\varepsilon}\subset\mathcal{A}$ valued ordinary
control on the interval $[i\varepsilon,i\varepsilon+\varepsilon)$,
and a solution $\boldsymbol{x}^{\eta}(\cdot)$ such that $\left|J(t,\boldsymbol{x}^{\eta},\gamma^{\eta})-J(t,\boldsymbol{x},\gamma)\right|\leq\eta$.
There is a $\vartheta$ such that the approximating $\boldsymbol{\alpha}^{\eta}(\cdot)$
can be chosen so that its probability law at $n\varepsilon$, conditioned
on $\left\{ \boldsymbol{W}^{\eta}(\tau),\tau\leq n\varepsilon;\boldsymbol{\alpha}^{\eta}(i\varepsilon),i<n\right\} $
depends only on the samples $\left\{ \boldsymbol{W}^{\eta}(p\vartheta),p\vartheta\leq n\varepsilon;\boldsymbol{\alpha}^{\eta}(i\varepsilon),i<n\right\} $,
and is continuous in the $\boldsymbol{W}^{\varepsilon}(p\vartheta)$ arguments.
\end{lemma}
With Lemma \ref{lemma1}, we have the following two theorems, whose proofs are presented in Appendix C and D.
\begin{thm}\label{Theorem 5}
Let the approximating chain $
\{ \boldsymbol{X}_{n}^{h,m_{h}},n<\infty
\} $
constructed with transition probabilities defined in (\ref{eq:3-3})
be locally consistent with (\ref{eq:2-1-2}).
Let $\{\boldsymbol{\alpha}_{n}^{h,m_{h}},n<\infty\}$ be sequences
of admissible controls and $\boldsymbol{\gamma}^{h,m_{h}}(\cdot)$
be the relaxed control representation of $\boldsymbol{\alpha}^{h,m_{h}}(\cdot)$
(continuous time interpolation of $\boldsymbol{\alpha}_{n}^{h,m_{h}}$).
Then the sequence $
(\boldsymbol{X}^{h,m_{h}},\boldsymbol{W}^{h,m_{h}},\gamma^{h,m_{h}})$ is tight, which has a weakly convergent subsequence denoted by $(\boldsymbol{X},\boldsymbol{W},\gamma)$.
In addition, $\lim_{h\rightarrow0}m_{h}=m$ in $\mathcal{P}^2(\mathcal{Q})$.
\end{thm}

\begin{thm}\label{Theorem 6}
$v(t,\boldsymbol{x})$ and $v^{h,m_{h}}(t,\boldsymbol{x})$ are value functions defined
in (\ref{eq:2-1-1}) and (\ref{eq:5-2-1-2}), respectively, we have
$v^{h,m_{h}}(t,\boldsymbol{x})\rightarrow v(t,\boldsymbol{x}), \text{as}\;h\rightarrow0.$
\end{thm}

\section{Numerical examples}\label{Numerical examples}
In this section, we will present two examples to illustrate the numerical method. The first example entails analytical solutions, as expounded in \cite{CrFjSl:13}, while the second delves into a 2-dimensional MFG scenario, extending upon the concepts outlined in \cite{BeBaCa:18}.
All methods are coded by Python with PyTorch package and run on a server with AMD EPYC 7T83 (64 cores) and RTX 4090 (24GB) GPU.
\subsection{Linear-quadratic MFGs with common noise}
Let $\left(W_{t}\right)_{0\leq t\leq T}$
and $\left(W_{t}^{0}\right)_{0\leq t\leq T}$ be independent Brownian
motions, which are referred as the idiosyncratic noise and common
noise, respectively. We consider linear-quadratic MFGs with common noise as follows,
\begin{equation*}
\setlength{\abovedisplayskip}{3pt}
\setlength{\belowdisplayskip}{3pt}
\begin{aligned} & \inf_{\alpha}\ep \Big\{
\int_{0}^{T}
\Big[\frac{\alpha_{t}^{2}}{2}-q\alpha_{t}
(u_{t}-X_{t})
+\frac{\epsilon}{2}
(u_{t}-X_{t}
)^{2}
\Big]dt
\\
 & \hspace{2in} 
 +\frac{c}{2}\left(u_{T}-X_{T}\right)^{2}
 \Big\},
\end{aligned}
\end{equation*}
where $X_{0}
\sim\mu_{0}$ and
\begin{equation*}
\setlength{\abovedisplayskip}{3pt}
\setlength{\belowdisplayskip}{3pt}
\begin{aligned}
\mathrm{d}X_{t} & =[a(u_{t}-X_{t})+\alpha_{t}]\mathrm{d}t+\varSigma(\rho
\mathrm{d}W_{t}^{0}+\sqrt{1-\rho^{2}}dW_{t}).
\end{aligned}
\end{equation*}
Here $u_{t}=\ep[X_{t}\mid\mathcal{F}_{t}^{W^{0}}]$ is the conditional population mean given the common noise. Then, at
equilibrium, we have
\begin{equation*}
\setlength{\abovedisplayskip}{3pt}
\setlength{\belowdisplayskip}{3pt}
\begin{aligned}
& u_{t}:=\ep\left[X_{0}\right]+\rho\varSigma W_{t}^{0},\quad t\in[0,T],\\
 & \alpha_{t}=\left(q+\eta_{t}\right)\left(u_{t}-X_{t}\right),\quad t\in[0,T],
\end{aligned}
\end{equation*}
where $\eta_{t}$ is a deterministic function of time solving the Riccati equation,
\begin{equation*}
\setlength{\abovedisplayskip}{3pt}
\setlength{\belowdisplayskip}{3pt}
\dot{\eta}_{t}=2(a+q)\eta_{t}+\eta_{t}^{2}-(\epsilon-q^{2}),\quad\eta_{T}=c,
\end{equation*}
with the solution given by
\begin{equation*}
\setlength{\abovedisplayskip}{3pt}
\setlength{\belowdisplayskip}{3pt}
\begin{aligned}
\eta_{t}= & \frac{-\left(\epsilon-q^{2}\right)\left(e^{\left(\delta^{+}-\delta^{-}\right)(T-t)}-1\right)}{\left(\delta^{-}e^{\left(\delta^{+}-\delta^{-}\right)(T-t)}-\delta^{+}\right)-c\left(e^{\left(\delta^{+}-\delta^{-}\right)(T-t)}-1\right)}\\
- & \frac{c\left(\delta^{+}e^{\left(\delta^{+}-\delta^{-}\right)(T-t)}-\delta^{-}\right)}{\left(\delta^{-}e^{\left(\delta^{+}-\delta^{-}\right)(T-t)}-\delta^{+}\right)-c\left(e^{\left(\delta^{+}-\delta^{-}\right)(T-t)}-1\right)}.
\end{aligned}
\end{equation*}
Here $\delta^{\pm}=-(a+q)\pm\sqrt{R},R=(a+q)^{2}+\left(\epsilon-q^{2}\right)>0$,
and the minimal expected cost for a representative player is $v\left(0,x_{0}-\ep[x_{0}]\right)$
with
\begin{equation*}
\setlength{\abovedisplayskip}{3pt}
\setlength{\belowdisplayskip}{3pt}
v(t, x)=\frac{\eta_{t}}{2}x^{2}+\frac{1}{2}\varSigma^{2}\left(1-\rho^{2}\right)\int_{t}^{T}\eta_{s}\mathrm{d}s.
\end{equation*}
In this experiment, the specified parameter values are configured as stated below (see \cite{HjHr:20}),
\begin{equation*}
\setlength{\abovedisplayskip}{3pt}
\setlength{\belowdisplayskip}{3pt}
\begin{array}{lccccccc}
\hline \text { Parameter } & {a} & {q} & {c} & {\epsilon} & {\rho} & {\varSigma} & {T}\\
\hline \text { Value } & 0.1 & 0.1 & 0.5 & 0.5 & 0.2 & 1 & 1 \\
\hline
\end{array}
\end{equation*}
Here, we partition the time interval $[0, T]$ into $100$ equal segments, i.e., $t_k=\frac{k}{100}, k=0,1, \ldots, 100$. The initial states are independently generated as $X_0$ follows the uniform distribution $\mu_0 = U(0, 1)$.

Figure \ref{fig:trajectories} \subref{fig:Xt}-\subref{fig:alpha} presents three sample paths (represented by RGB: Red, Green, Blue) for the optimal state process $X_{t}$, the conditional mean $u_{t}$ and the optimal control $\alpha_{t}$ vs. their approximations $\hat{X}_{t}$, $\hat{u}_{t}$, $\hat{\alpha}_{t}$ (represented by dashed lines with circles) provided by our numerical algorithm. Obviously, all the state (Figure \ref{fig:trajectories} \subref{fig:Xt}), mean-field interaction (Figure \ref{fig:trajectories} \subref{fig:ut}) and control (Figure \ref{fig:trajectories} \subref{fig:alpha}) trajectories are aligned with the analytical solution (represented by solid lines).

\begin{figure}[http] 	\centering
	\vspace{-0.3cm}
	\subfigtopskip=2pt
	\subfigbottomskip=2pt
	\subfigcapskip=-5pt
	\subfigure[$X_{t}$]{
		\label{fig:Xt}
		\includegraphics[width=0.45\linewidth]{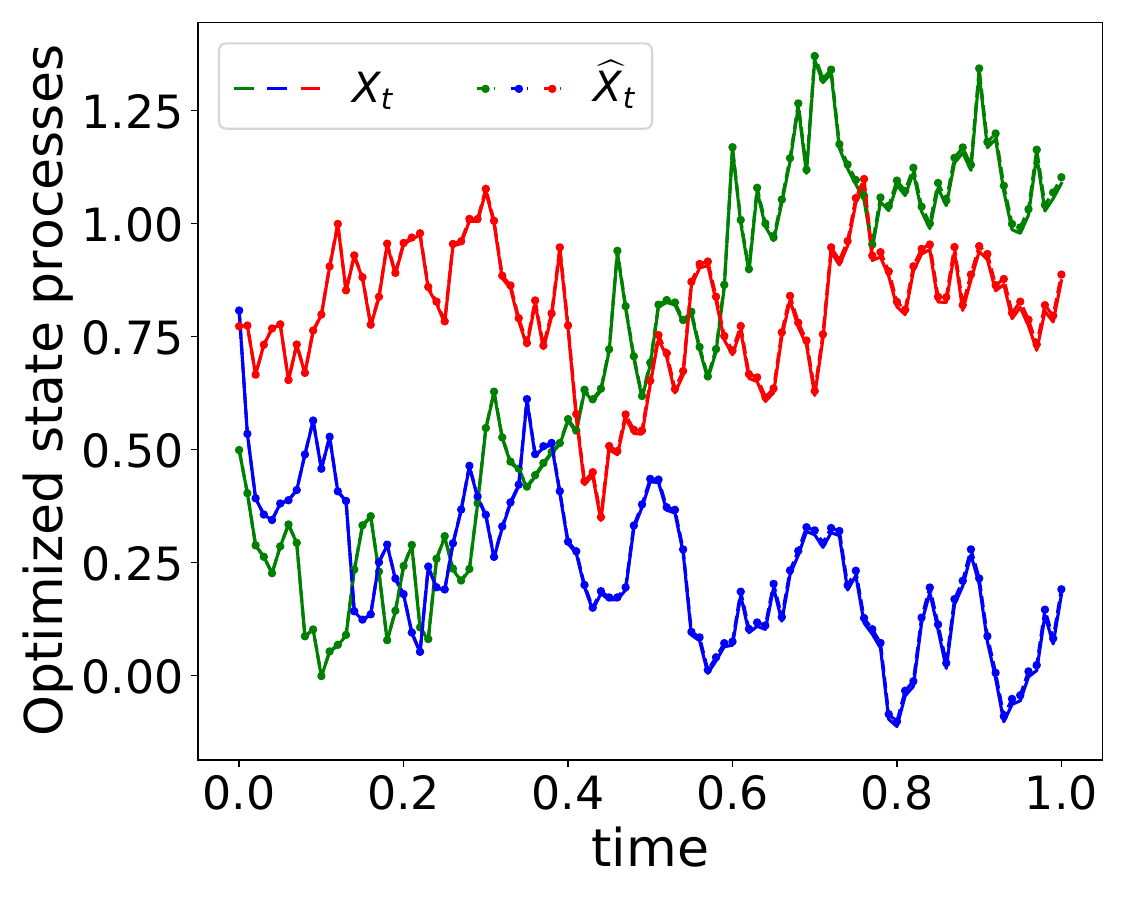}}
		\subfigure[$u_{t}$]{
		\label{fig:ut}
		\includegraphics[width=0.45\linewidth]{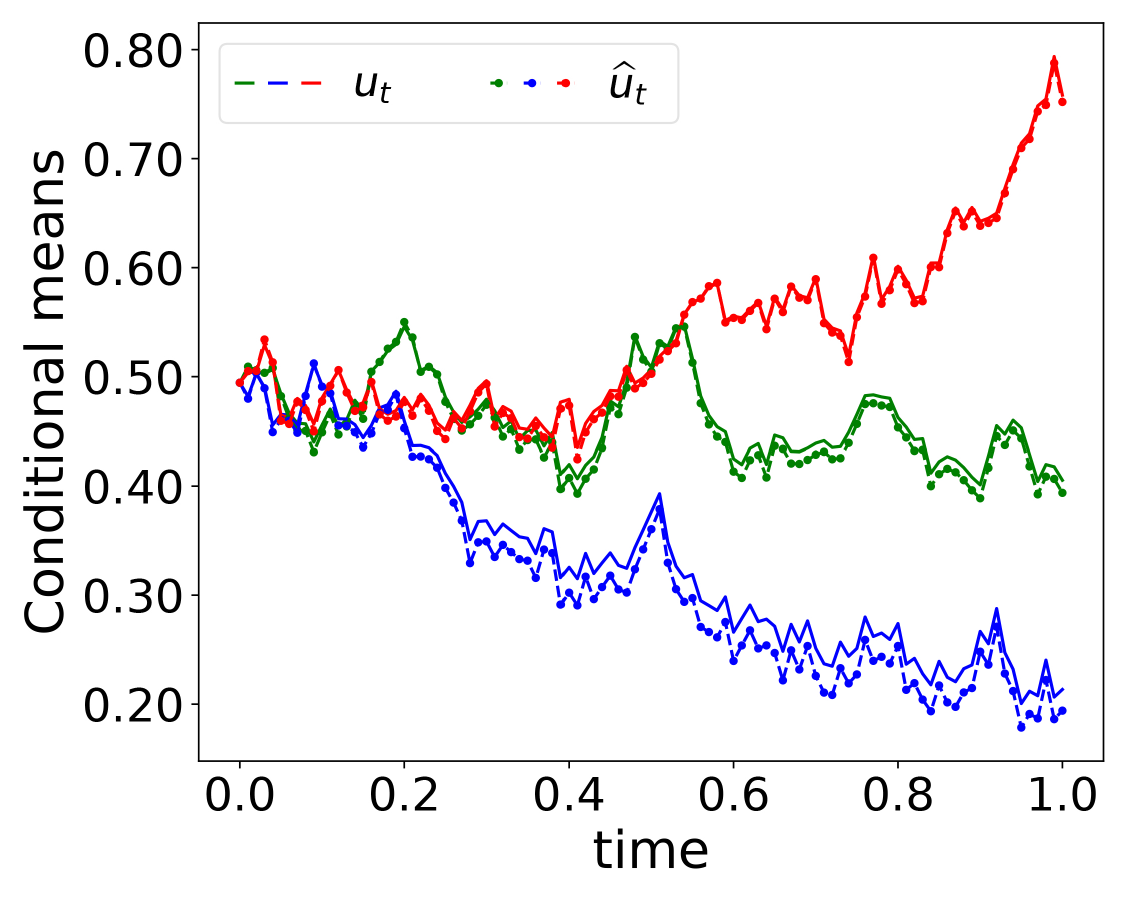}}
	\subfigure[$\alpha_{t}$]{
		\label{fig:alpha}
		\includegraphics[width=0.45\linewidth]{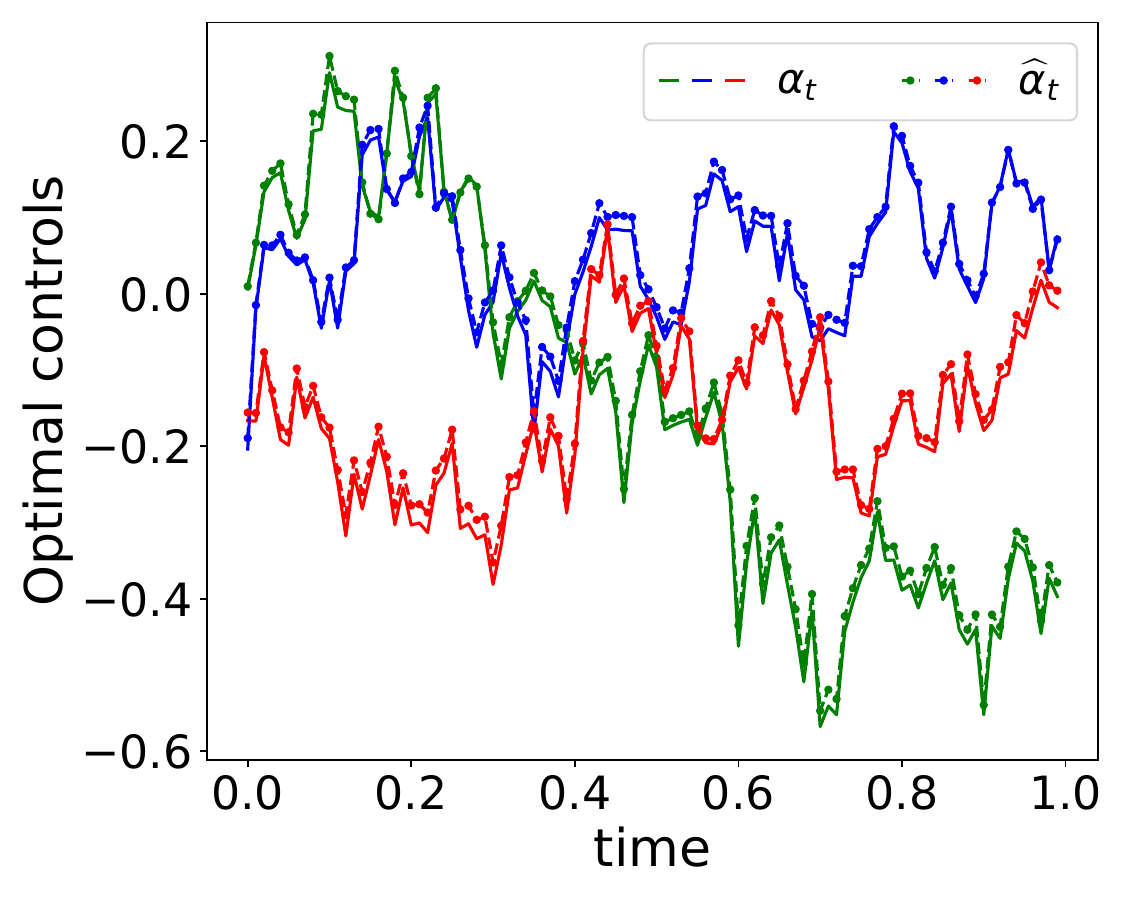}}
		\caption{Figures (a), (b) and (c) give three trajectories of $X_{t}$, $u_{t}$ and $\alpha_{t}$ (solid lines) and their approximations (circles).}
		\label{fig:trajectories}
\end{figure}

%

\subsection{$2$-dimensional MFGs}
Suppose that the state variable $\boldsymbol{x}=(x_{1},x_{2})^{\top}\in\mathcal{Q}$,
and the control variable $\boldsymbol{\alpha}=(\alpha_{1},\alpha_{2})^{\top}\in\boldsymbol{U}$.
We first give the parameters of deep learning, such as the computation precision $\epsilon$ and maximal numbers
of learning times for locating the initial value of $\theta$, for
determining iterative control strategy, and for iteratively updating
value function $v_{k}$. The sets of this experiment are as follows (see \cite{CxJzYh:20}):
\begin{equation*}
\setlength{\abovedisplayskip}{3pt}
\setlength{\belowdisplayskip}{3pt}	
\begin{array}{lcc}
\hline & \text { Triggering Error } & \text { Max \# of steps } \\
\hline \text { Control Fit } & 10^{-3} & 10000 \\
\text { Gradient Descent } & 10^{-5} & 5000 \\
\text { Global Iteration } & 10^{-6} & 50000 \\
\hline
\end{array}
\end{equation*}
Next, we set the parameters related to the MCAM method (see \cite{BeBaCa:18}) as follows: $h_{1}=0.2$, $h_{2}=0.01$, $T=1$, $\varSigma=0.5$,
$\mathcal{Q}=[0,1]^{2}$, and $\boldsymbol{U}=[0,1.5]^{2}$.
To be simplicity, we assume that the state equation does not involve the mean-field interaction, and the form of the drift $\boldsymbol{b}$ is given by: $\boldsymbol{b}(t,\boldsymbol{x},u,\boldsymbol{\alpha})=2\boldsymbol{x}-\boldsymbol{\alpha}$. Meanwhile, assuming that the mean-field interaction only appears in the objective function, and the form of the running and terminal cost functions are
\begin{equation*}
\setlength{\abovedisplayskip}{3pt}
\setlength{\belowdisplayskip}{3pt}
\begin{aligned}
f(t,\boldsymbol{x},u,\boldsymbol{\alpha}) & =\Vert 4\boldsymbol{x}-5\bar{u}\Vert ^{2}+\Vert \boldsymbol{\alpha}\Vert ^{2},\\
g(\boldsymbol{x},u) & =\Vert 4\boldsymbol{x}-5\bar{u}\Vert ^{2},
\end{aligned}
\end{equation*}
where $\bar{u}$ is the mean of $u$.

In this set of experiments, we choose the initial state of the MFGs is taken to be $\boldsymbol{x}_{0}=(0.4,0.4)^{\top}$, and assume that agents are initially distributed according to
\begin{equation*}
\setlength{\abovedisplayskip}{3pt}
\setlength{\belowdisplayskip}{3pt}
u_{0}(\boldsymbol{x})\propto\exp(-\frac{1}{2}(\mathbf{x}-\boldsymbol{\mu})^{\mathrm{T}}\boldsymbol{\varSigma}^{-1}(\mathbf{x}-\boldsymbol{\mu})),\;\boldsymbol{x}\in\mathbb{R}^{2},
\end{equation*}
where $\boldsymbol{\mu}=(0,1)\in\mathbb{R}^{2}$, and $\boldsymbol{\varSigma}=0.25\mathbb{I}_{2\times2}$, where $\mathbb{I}_{2\times2}$ is a $2$-dimensional identity matrix.

Fix $t=0.5$, the relationship between the value function $v\left(0.5, \boldsymbol{x}\right)$ and the state $\boldsymbol{x}=\left(x_1, x_2\right)$ is shown in Figure \ref{fig:u}. When $x_{2}$ is fixed, be it $0$, $0.5$ or $1$, the value function $v(0.5, x_1)$ increases as the value of $x_{1}$ increases (see Figure \ref{fig:u_alpha_x12} \subref{fig1}). The same positive relation between the value function $v(0.5, x_2)$ when $x_{1}$ is fixed also applies (see Figure \ref{fig:u_alpha_x12} \subref{fig2}). Figure \ref{fig:u_alpha_x12} \subref{fig3}-\subref{fig4} shows that the relationship between optimal controls and the state $\boldsymbol{x}$. In Figure \ref{fig:w}, shortly after $t=0$, the means tend to increase w.r.t. the time. The reason for this is that the value of optimal controls are positive numbers so that the process $X$ is elevated until reaching the upper limit of $1$. Regarding the optimal control curve's values in Figure \ref{fig:u_alpha_x12}, a higher value of $\alpha$ leads to faster attainment of the maximum mean value in Figure \ref{fig:w}.
\begin{figure}[htbp]
\vspace{-0.3cm}
\centerline{\includegraphics[width=0.7\linewidth]{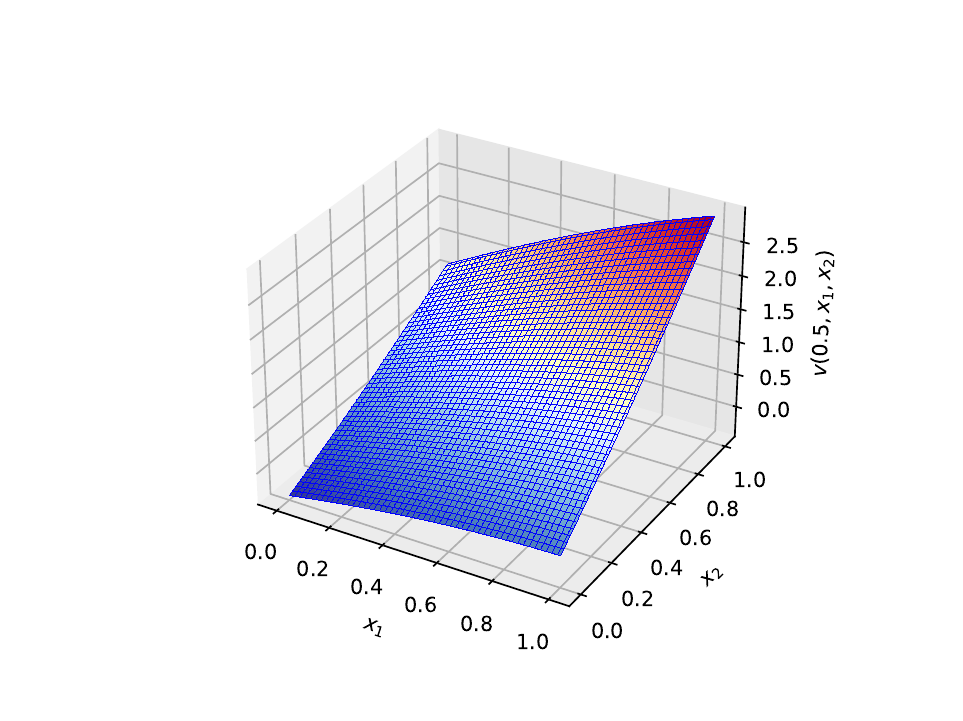}}
\caption{Value function vs. $x_{1}$ and $x_{2}$.}
\label{fig:u}
\end{figure}

\begin{figure}[http] 	\centering
	\vspace{-0.3cm}
	\subfigtopskip=2pt
	\subfigbottomskip=2pt
	\subfigcapskip=-5pt
	\subfigure[]{
		\label{fig1}
		\includegraphics[width=0.45\linewidth]{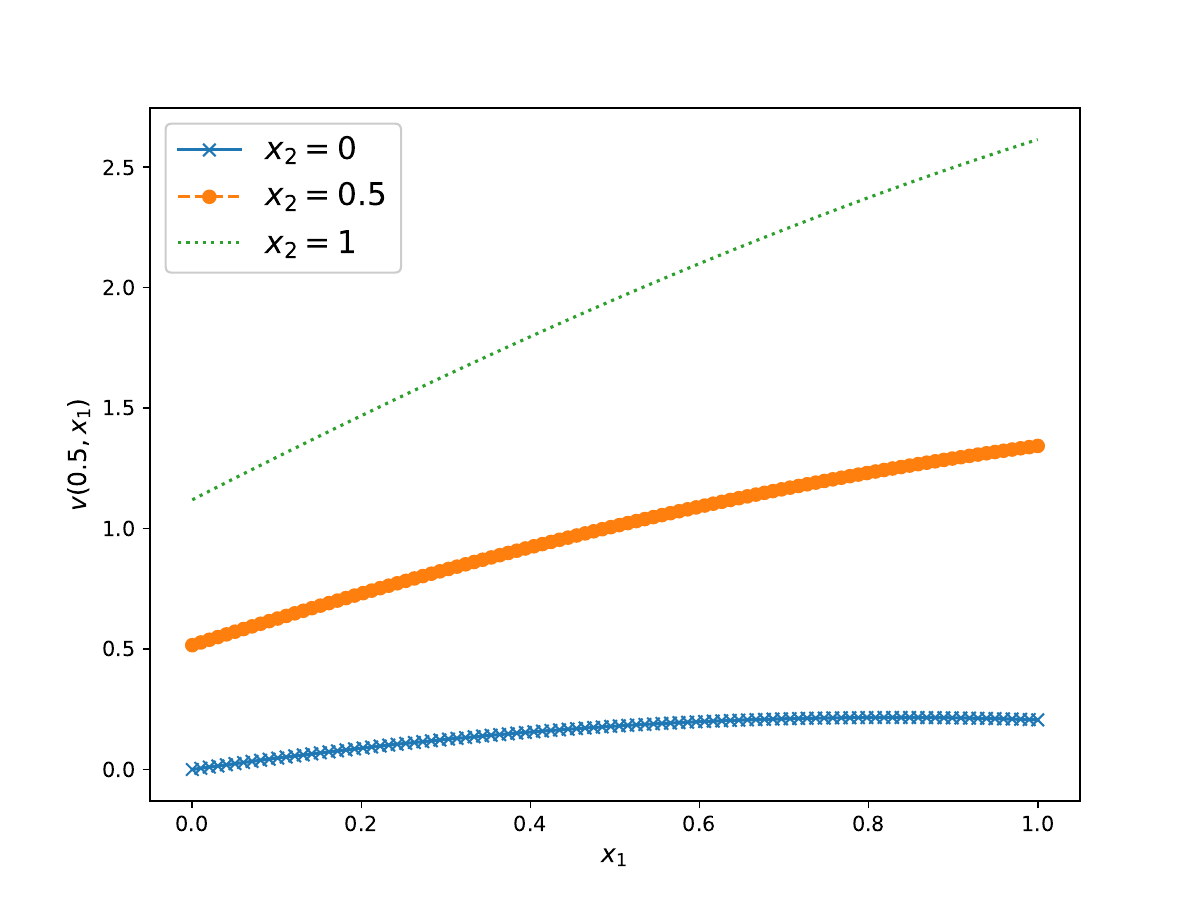}}
	\hspace{-4mm}
	\subfigure[]{
		\label{fig2}
		\includegraphics[width=0.45\linewidth]{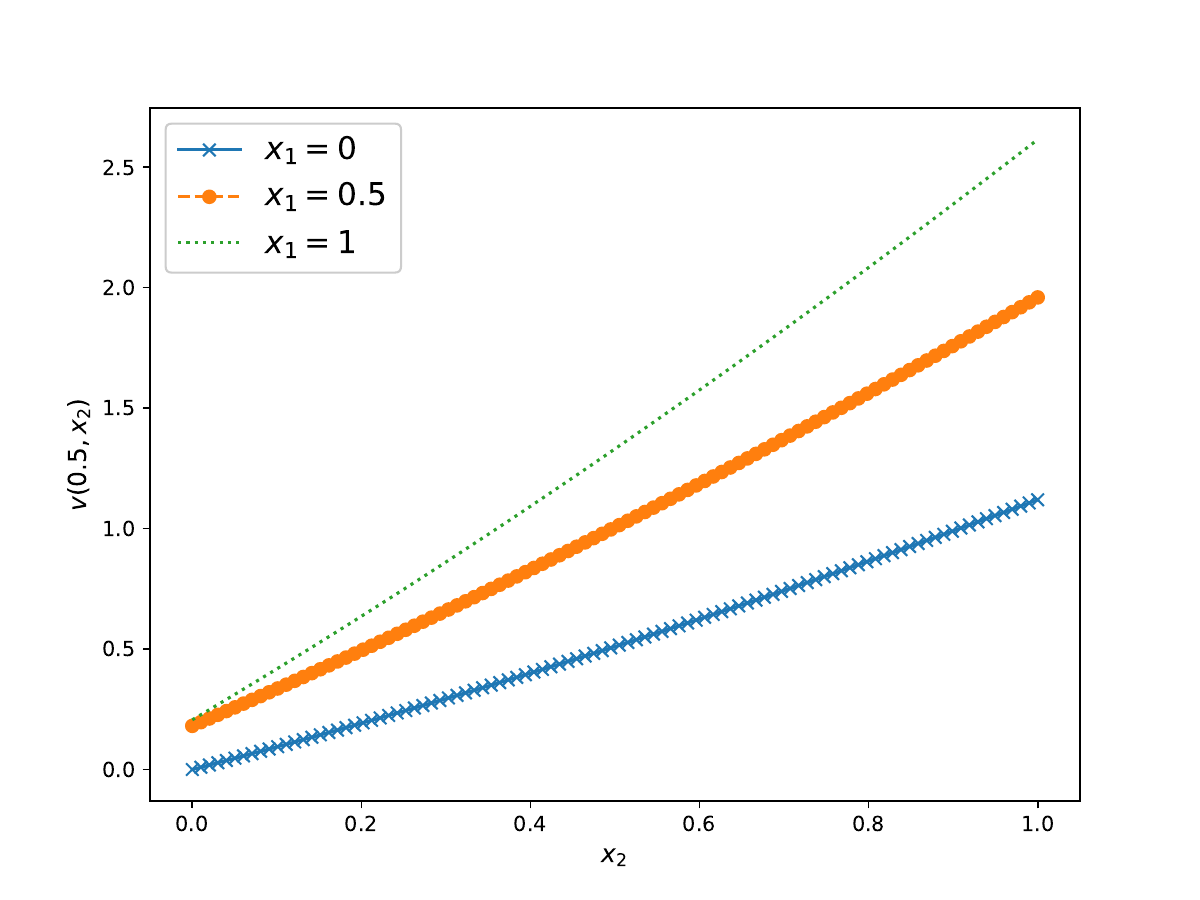}}
		\subfigure[]{
		\label{fig3}
		\includegraphics[width=0.45\linewidth]{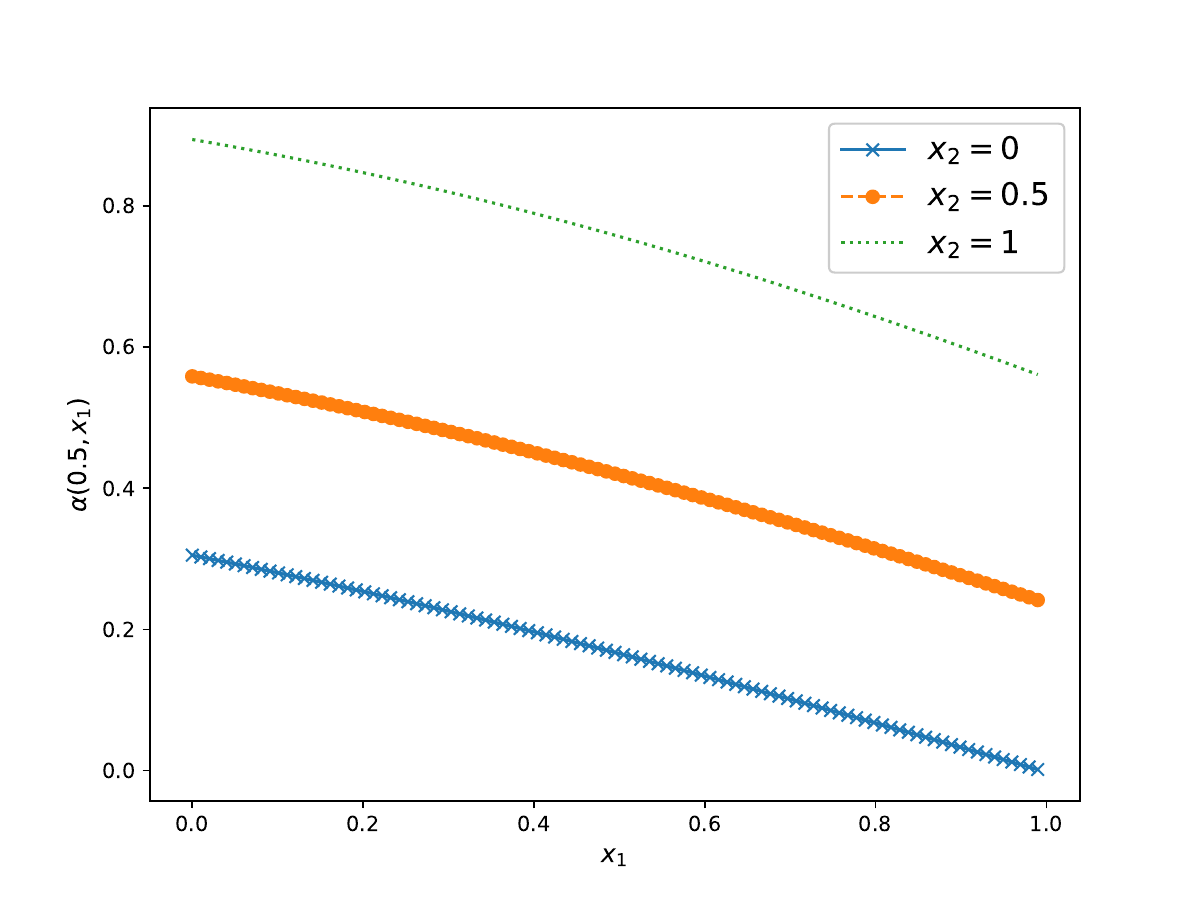}}
	\hspace{-4mm}
	\subfigure[]{
		\label{fig4}
		\includegraphics[width=0.45\linewidth]{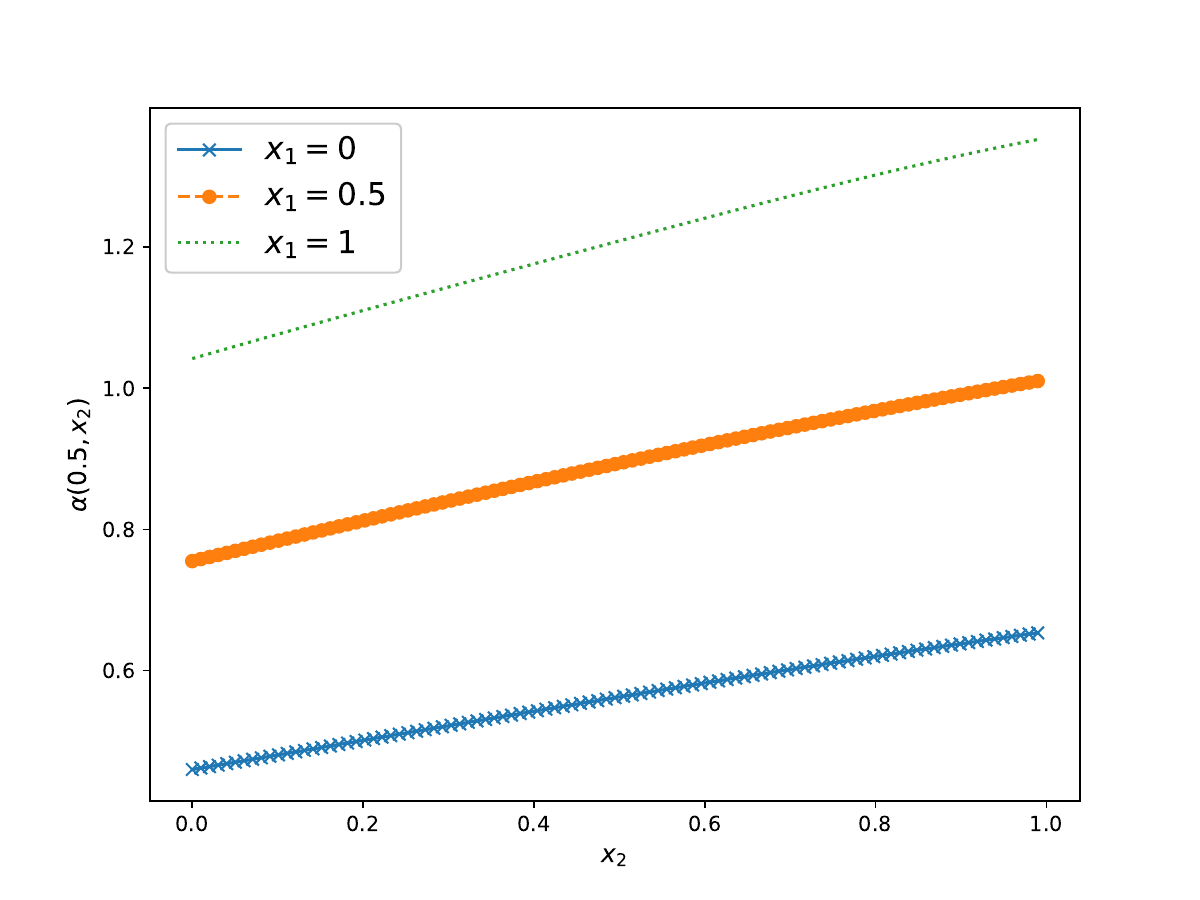}}
		\caption{The optimal control and value function.}
		\label{fig:u_alpha_x12}
\end{figure}

\begin{figure}[http] 	\centering
	\vspace{-0.3cm}
	\subfigtopskip=2pt
	\subfigbottomskip=2pt
	\subfigcapskip=-5pt
	\subfigure[]{
		\label{fig5}
		\includegraphics[width=0.45\linewidth]{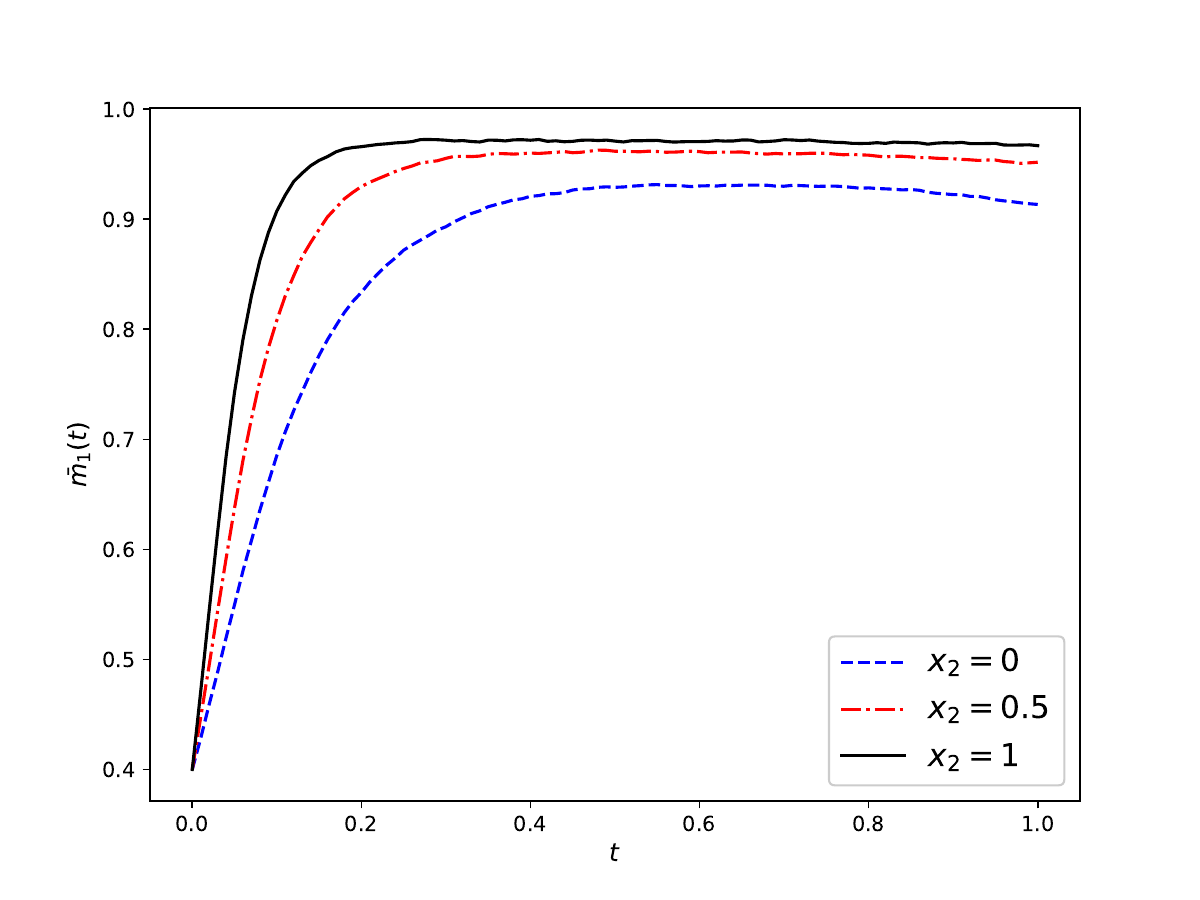}}
	\hspace{-4mm}
	\subfigure[]{
		\label{fig6}
		\includegraphics[width=0.45\linewidth]{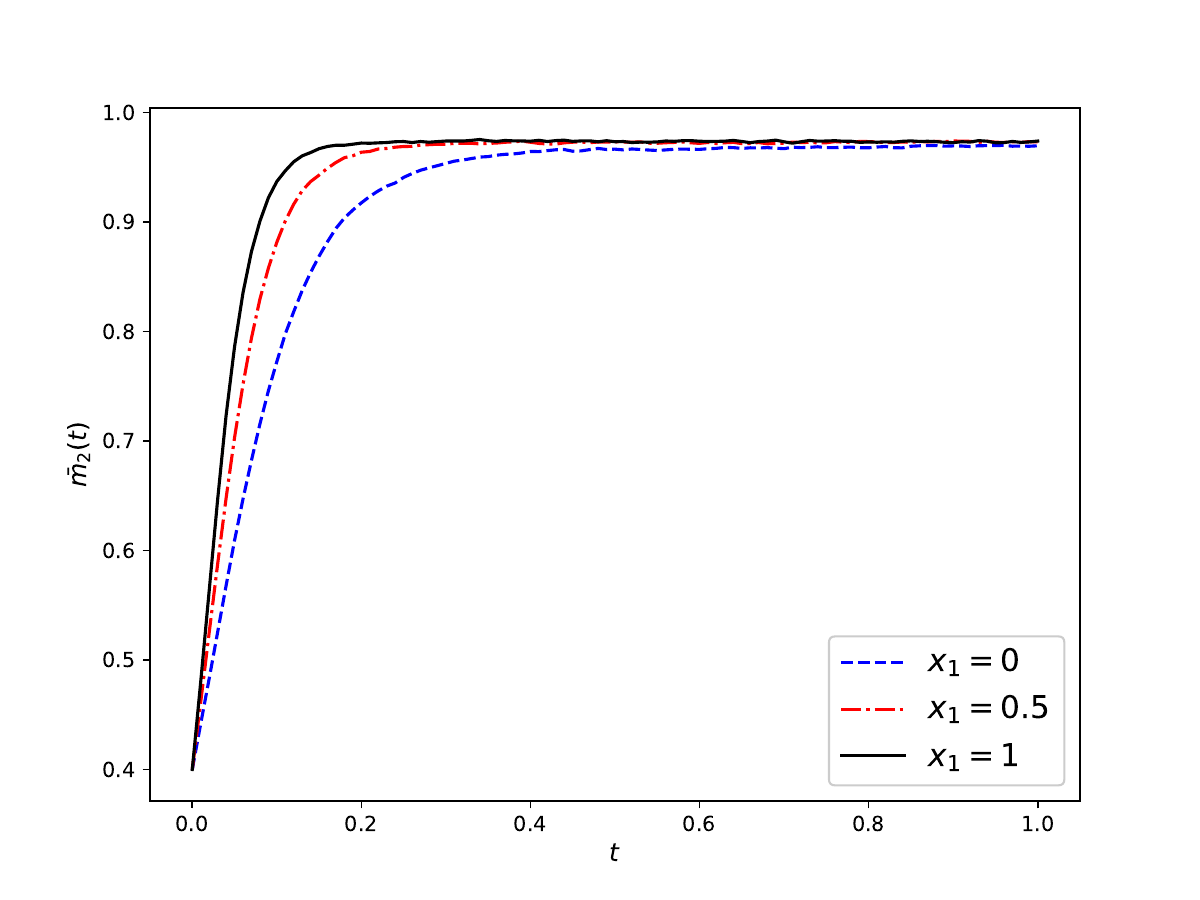}}
	\caption{The means $\bar{u}(t)$ given values $x_{1}$ and $x_{2}$.}
	\label{fig:w}
\end{figure}

\section{Conclusion}\label{Conclusion}
This paper developed a hybrid Markov chain approximation-based deep learning algorithm to solve the finite-horizon MFGs. The induced measure is adopted to build an iterative mean-field interaction in MCAM, whereas Monte-Carlo algorithm is 
used  for the approximation of the mean-field interaction.
The convergence of the hybrid algorithm is established
and two numerical examples are used to demonstrate the performance of our algorithm.

On one hand, the approximating piecewise constant controls
obtained by MCAM method depend on the density of the control grid, which determines the accuracy of the controls but is limited by the computational capability. On the other hand, the optimal control obtained through the neural network may be a solution that is locally optimal. To address this, our proposed method combines the advantages of the aforementioned two methods to find more accurate control values. In addition, our approach can 
overcome the difficulty of
curse of dimensionality
 that existing numerical methods are suffering.

In the future studies, the problem of MFC can be further addressed by using our proposed algorithm in this paper. Moreover, given the fact that the efficiency of Monte-Carlo algorithm remains relatively low, alternative approaches may be further developed. Apart from the general recursion that $\theta_{k+1}=\theta_k+\rho \nabla G\left(\theta_k, \epsilon_k\right)$ has been widely used in SA algorithm, we may also propose a modified recursion to adaptively approximate the optimal learning rate in our hybrid algorithm; see \cite{QhYgZq:22} for related work.

\begin{ack}                               
The research of Jiaqin Wei was supported in part by the National Natural Science
Foundation of China (12071146). 
\end{ack}

\appendix
\section{Proof of Equation (\ref{eq:3-3})}\label{Eq:3-3}
\begin{proof}
To prove equation (\ref{eq:3-3}), the results obtained from the finite difference approach are substituted into equation (\ref{eq:2-3-1}), resulting in the derived expression. Note that the variables $t, \boldsymbol{x}, m$, and $\hat{\boldsymbol{\alpha}}$ associated with $b^{i \pm}$ are omitted for brevity.
\begin{equation*}
\setlength{\abovedisplayskip}{3pt}
\setlength{\belowdisplayskip}{3pt}
\begin{aligned} & \frac{v(nh_{2}+h_{2},\boldsymbol{x})-v(nh_{2},\boldsymbol{x})}{h_{2}}\\
 & +\sum_{i}\left[b^{i+}\frac{v(nh_{2}+h_{2},\boldsymbol{x}+h_{1}\boldsymbol{e}_{i})-v(nh_{2}+h_{2},\boldsymbol{x})}{h_{1}}\right.\\
 & \quad\left.-b^{i-}\frac{v(nh_{2}+h_{2},\boldsymbol{x})-v(nh_{2}+h_{2},\boldsymbol{x}-h_{1}\boldsymbol{e}_{i})}{h_{1}}\right]\\
 & +\frac{1}{2}\sum_{i}\sum_{j\neq i}a_{ij}^{+}\left[\frac{2v(nh_{2}+h_{2},\boldsymbol{x})}{2h_{1}^{2}}\right.\\
 & \quad+\frac{v(nh_{2}+h_{2},\boldsymbol{x}+h_{1}\boldsymbol{e}_{i}+h_{1}\boldsymbol{e}_{j})}{2h_{1}^{2}}\\
 & \quad-\frac{v(nh_{2}+h_{2},x+h_{1}\mathbf{e}_{j})+v(nh_{2}+h_{2},x-h_{1}\mathbf{e}_{j})}{2h_{1}^{2}}\\
 & \quad+\frac{v(nh_{2}+h_{2},\boldsymbol{x}-h_{1}\mathbf{e}_{i}-h_{1}\mathbf{e}_{j})}{2h_{1}^{2}}\\
 & \quad\left.-\frac{v(nh_{2}+h_{2},\boldsymbol{x}+h_{1}\mathbf{e}_{i})+v(nh_{2}+h_{2},x-h_{1}\mathbf{e}_{i})}{2h_{1}^{2}}\right]\\
 & -\frac{1}{2}\sum_{i}\sum_{j\neq i}a_{ij}^{-}\left[-\frac{2v(nh_{2}+h_{2},\boldsymbol{x})}{2h_{1}^{2}}\right.\\
 & \quad+\frac{v(nh_{2}+h_{2},\boldsymbol{x}+h_{1}\boldsymbol{e}_{i}-h_{1}\boldsymbol{e}_{j})}{2h_{1}^{2}}\\
 & \quad+\frac{v(nh_{2}+h_{2},x+h_{1}\mathbf{e}_{j})+v(nh_{2}+h_{2},x-h_{1}\mathbf{e}_{j})}{2h_{1}^{2}}\\
 & \quad+\frac{v(nh_{2}+h_{2},\boldsymbol{x}+h_{1}\mathbf{e}_{i})+v(nh_{2}+h_{2},x-h_{1}\mathbf{e}_{i})}{2h_{1}^{2}}\\
 & \quad\left.-\frac{v(nh_{2}+h_{2},\boldsymbol{x}-h_{1}\mathbf{e}_{i}+h_{1}\mathbf{e}_{j})}{2h_{1}^{2}}\right]\\
 & +\frac{1}{2}\sum_{i}a_{ii}\left[\frac{v(nh_{2}+h_{2},\boldsymbol{x}+h_{1}\boldsymbol{e}_{i})}{h_{1}^{2}}\right.\\
 & \quad\left.+\frac{v(nh_{2}+h_{2},\boldsymbol{x}-h_{1}\boldsymbol{e}_{i})-2v(nh_{2}+h_{2},\boldsymbol{x})}{h_{1}^{2}}\right]\\
 & +f(nh_{2}+h_{2},\boldsymbol{x},m,\hat{\boldsymbol{\alpha}})\\
= & 0.
\end{aligned}
\end{equation*}
By simplifying and rearranging the terms of the above equation, we can get
\begin{equation*}
\setlength{\abovedisplayskip}{3pt}
\setlength{\belowdisplayskip}{3pt}
\begin{aligned}\\
 & v(nh_{2},\boldsymbol{x})=v(nh_{2}+h_{2},\boldsymbol{x})\left(1-\sum_{i}\frac{\left(a_{ii}+|b_{i}|h_{1}\right)h_{2}}{h_{1}^{2}}\right)\\
+ & \sum_{i}v(nh_{2}+h_{2},\boldsymbol{x}\pm h_{1}\boldsymbol{e}_{i})\frac{(a_{ii}/2-\sum_{j\neq i}|a_{ij}|/2+b_{i}^{\pm}h_{1})h_{2}}{h_{1}^{2}}\\
+ & \sum_{i}\sum_{j\neq i}v(nh_{2}+h_{2},\boldsymbol{x}+h_{1}\boldsymbol{e}_{i}\pm h_{1}\boldsymbol{e}_{j})\frac{a_{ij}^{\pm}h_{2}}{2h_{1}^{2}}\\
+ & \sum_{i}\sum_{j\neq i}v(nh_{2}+h_{2},\boldsymbol{x}-h_{1}\boldsymbol{e}_{i}\pm h_{1}\boldsymbol{e}_{j})\frac{a_{ij}^{\pm}h_{2}}{2h_{1}^{2}}\\
+ & f(nh_{2}+h_{2},\boldsymbol{x},m,\hat{\boldsymbol{\alpha}})
\end{aligned}
\end{equation*}
By analyzing the iterative formula presented above, we can derive the transition probabilities (\ref{eq:3-3}) for the Markov chains utilized in MFGs.
\end{proof}
\section{Proof of Lemma \ref{lemma1}} \label{Chattering Lemma}
\begin{proof}
We divide the proof into two parts as follows.

Part 1. For any $\eta>0$, there are $\varepsilon>0$, a finite set
$\mathcal{A}^{\eta}\subset\mathcal{A}$, and a probability space on
which are defined a solution in the stochastic differential equation
in (\ref{eq:2-1-2}). Thus, we have $\left(\boldsymbol{X}^{\eta}(\cdot),\boldsymbol{\alpha}^{\eta}(\cdot),\boldsymbol{W}^{\eta}(\cdot)\right)$,
where $\boldsymbol{\alpha}^{\eta}(\cdot)$ is $\mathcal{A}^{\eta}$-valued
and constant on $[n\varepsilon,n\varepsilon+\varepsilon)$. Moreover,
$\left(\boldsymbol{X}^{\eta}(\cdot),\gamma^{\eta}(\cdot),\boldsymbol{W}^{\eta}(\cdot)\right)$
converges weakly to $\left(\boldsymbol{X}(\cdot),\gamma(\cdot),\boldsymbol{W}(\cdot)\right)$,
which further implies that 
\begin{equation*}	
\limsup_{\eta}\left|J(t,\boldsymbol{x},\gamma^{\eta})-J(t,\boldsymbol{x},\gamma)\right|\leq\varepsilon_{\tilde{\eta}},	
\end{equation*}
with $\varepsilon_{\tilde{\eta}}$ satisfying $\varepsilon_{\tilde{\eta}}\rightarrow0$
as $\tilde{\eta}\rightarrow0$.

Part 2. Consider a $\boldsymbol{\alpha}^{\eta}(\cdot)$ as in Part
1 above for $\eta$ sufficiently small. Let $0<\vartheta<\varepsilon$.
For $\boldsymbol{\pi}\in\mathcal{A}^{\eta}$, define the function
$F_{n,\vartheta}$ as the regular conditional probability 
\begin{equation*}
\begin{aligned}
 & F_{n,\vartheta}\left(\boldsymbol{\pi};\boldsymbol{\alpha}^{\eta}(i\varepsilon),i<n,\boldsymbol{W}^{\eta}(p\vartheta),p\vartheta\leq n\varepsilon\right)\\
= & \pr\left(\boldsymbol{\alpha}^{\eta}(n\varepsilon)=\boldsymbol{\pi}\mid\boldsymbol{\alpha}^{\eta}(i\varepsilon),i<n,\boldsymbol{W}^{\eta}(p\vartheta),p\vartheta\leq n\varepsilon\right).
\end{aligned}
\end{equation*}
The uniqueness of the solution of the state equation or the associated
martingale problem implies that the law of $\left(\boldsymbol{X}^{\eta},\gamma^{\eta}(\cdot),\boldsymbol{W}^{\eta}(\cdot)\right)$
is determined by the law of $\left(\gamma^{\eta}(\cdot),\boldsymbol{W}^{\eta}(\cdot)\right)$.
Since the $\sigma$-algebra determined by $\left\{ \boldsymbol{\alpha}^{\eta}(i\varepsilon),i<n,\boldsymbol{W}^{\eta}(p\vartheta),p\vartheta\leq n\varepsilon\right\} $
increases to the $\sigma$-algebra determined by $\left\{ \boldsymbol{\alpha}^{\eta}(i\varepsilon),i<n,\boldsymbol{W}^{\eta}(\tau),\tau\leq n\varepsilon\right\} $
as $\vartheta\rightarrow0$, we can show that, for each $n,\boldsymbol{\pi}$,
and $\varepsilon$, 
\begin{equation}
\begin{aligned}	
 & F_{n,\vartheta}\left(\boldsymbol{\pi};\boldsymbol{\alpha}^{\eta}(i\varepsilon),i<n,\boldsymbol{W}^{\eta}(p\vartheta),p\vartheta\leq n\varepsilon\right) \\
\rightarrow & \pr\left(\boldsymbol{\alpha}^{\eta}(n\varepsilon)=\boldsymbol{\pi}\mid\boldsymbol{\alpha}^{\eta}(i\varepsilon),i<n,\boldsymbol{W}^{\eta}(\tau),\tau\leq n\varepsilon\right)\label{eq:A.1}
\end{aligned}	
\end{equation}
with probability one as $\vartheta\rightarrow0$.

For $\boldsymbol{W}^{\eta,\vartheta}(\cdot)$, define the control
$\boldsymbol{\alpha}^{\eta,\vartheta}(\cdot)$ by the conditional
probability given in the first line of (\ref{eq:A.1}) with $\eta$
replaced by $\eta,\vartheta$. Owing to the construction of the control
law, as $\vartheta\rightarrow0,\left(\gamma^{\eta,\vartheta}(\cdot),\boldsymbol{W}^{\eta,\vartheta}(\cdot)\right)$
converges weakly to $\left(\gamma^{\eta}(\cdot),\boldsymbol{W}^{\eta}(\cdot)\right)$.
Thus, we can further show that $\left(\boldsymbol{X}^{\eta,\vartheta}(\cdot),\gamma^{\eta,\vartheta}(\cdot),\boldsymbol{W}^{\eta,\vartheta}(\cdot)\right)$
converges weakly to $\left(\boldsymbol{X}^{\eta}(\cdot),\gamma^{\eta}(\cdot),\boldsymbol{W}^{\eta}(\cdot)\right)$
as $\vartheta\rightarrow0$, and, moreover, $\boldsymbol{X}^{\eta,\vartheta}(\cdot)$
converges weakly to $\boldsymbol{X}(\cdot)$. Thus 
\begin{equation}
\left|J(t,\boldsymbol{x},\gamma^{\eta,\vartheta})-J(t,\boldsymbol{x},\gamma^{\eta})\right|\leq g_{1}(\vartheta),\label{eq:A.2}
\end{equation}
where $g_{1}(\vartheta)\rightarrow0$ as $\vartheta\rightarrow0$.

For $\Delta>0$, consider the ``mollified'' functions $F_{n,\vartheta,\Delta}(\cdot)$
given by
\begin{equation*}
\begin{aligned} 
& F_{n,\vartheta,\Delta}(\boldsymbol{\pi};\boldsymbol{\alpha}(i\varepsilon),i<n,\boldsymbol{W}(p\vartheta),p\vartheta\leq n\varepsilon)\\
= & N(\Delta)\int\cdots\int F_{n,\vartheta}\left(\boldsymbol{\pi};\boldsymbol{\alpha}(i\varepsilon),i<n,\boldsymbol{W}(p\vartheta)+\boldsymbol{z}_{p},p\vartheta\leq n\varepsilon\right)\\
 & \times\prod_{p}\exp\left(-\left\Vert \boldsymbol{z}_{p}\right\Vert ^{2}/(2\Delta)\right)\mathrm{d}\boldsymbol{z}_{p},
\end{aligned}
\end{equation*}
where $N(\Delta)$ is a normalizing constant so the integral of the
mollifier is unity. Note that $F_{n,\vartheta,\Delta}$ are nonnegative,
their values summing to unity, and they are continuous in the $\boldsymbol{W}$-variables.
As $\Delta\rightarrow0$, $F_{n,\vartheta,\Delta}$ converges to $F_{n,\vartheta}$
with probability one. 

Let $\boldsymbol{\alpha}^{\eta,\vartheta,\Delta}(\cdot)$ be the piecewise
constant admissible control that is determined by the conditional
probability distribution $F_{n,\vartheta,\Delta}(\cdot)$. There is
a probability space on which we can define $\boldsymbol{W}^{\eta,\vartheta,\Delta}(\cdot)$
and the control law $\boldsymbol{\alpha}^{\eta,\vartheta,\Delta}(\cdot)$
by the conditional probability 
\begin{equation*}
\begin{aligned}
 & \pr\left(\boldsymbol{\alpha}^{\eta,\vartheta,\Delta}(n\varepsilon)=\boldsymbol{\pi}\mid\boldsymbol{\alpha}^{\eta,\vartheta,\Delta}(i\varepsilon),i<n,\boldsymbol{W}^{\eta,\vartheta,\Delta}(\tau),\tau\leq n\varepsilon\right)\\
= & F_{n,\vartheta,\Delta}\left(\boldsymbol{\pi};\boldsymbol{\alpha}^{\eta,\vartheta,\Delta}(i\varepsilon),i<n,\boldsymbol{W}^{\eta,\vartheta,\Delta}(p\vartheta),p\vartheta\leq n\varepsilon\right).
\end{aligned}
\end{equation*}
Then the construction of the probability law of the controls, $\left(\boldsymbol{X}^{\eta,\vartheta,\Delta}(\cdot),\gamma^{\eta,\vartheta,\Delta}(\cdot),\boldsymbol{W}^{\eta,\vartheta,\Delta}(\cdot)\right)$
converges weakly to $\left(\boldsymbol{X}^{\eta,\vartheta}(\cdot),m^{\eta,\vartheta}(\cdot),\boldsymbol{W}^{\eta,\vartheta}(\cdot)\right)$
as $\Delta\rightarrow0$. This yields that 
\begin{equation}
\left|J(t,\boldsymbol{x},\gamma^{\eta,\vartheta,\Delta})-J(t,\boldsymbol{x},\gamma^{\eta,\vartheta})\right|\leq g_{2}(\Delta),\label{eq:A.3}
\end{equation}
where $g_{2}(\Delta)\rightarrow0$ as $\Delta\rightarrow0$.

Putting the above arguments together, and noting that $\tilde{\eta}$
can be chosen arbitrarily small. Then, for each $\eta>0$, there are
$\varepsilon>0,\vartheta>0,\boldsymbol{W}^{\eta}(\cdot)$, and an
admissible control that is piecewise constant on $[n\varepsilon,n\varepsilon+\varepsilon)$
taking values in a finite set $\mathcal{A}^{\eta}\subset\mathcal{A}$
determined by the conditional probability law 
\begin{equation*}
\begin{aligned} & \text{\pr\ensuremath{\left(\boldsymbol{\alpha}^{\eta}(n\varepsilon)=\boldsymbol{\pi}\mid\boldsymbol{\alpha}^{\eta}(i\varepsilon),i<n,\boldsymbol{W}^{\eta}(\tau),\tau\leq n\varepsilon\right)}}\\
= & F_{n}\left(\boldsymbol{\pi};\boldsymbol{\alpha}^{\eta}(i\varepsilon),i<n,\boldsymbol{W}^{\eta}(p\vartheta),p\vartheta\leq n\varepsilon\right),
\end{aligned}
\end{equation*}
where $F_{n}(\cdot)$ are continuous w.p. 1 in the $\boldsymbol{W}$-variables
for each of the other variables, and the proof is completed.
\end{proof}

\section{Proof of Theorem \ref{Theorem 5}}\label{Appendix 7.1}
\begin{proof}
Since $\gamma^{h,m_{h}}(\boldsymbol{U}\times[0,T])=T$
and $\boldsymbol{U}$ is compact, the sequence $\{\gamma^{h,m_{h}}\}_{h>0}$
is tight in $\mathscr{B}(\boldsymbol{U}\times[0,T])$. Here, we consider
the tightness of $\boldsymbol{X}^{h,m_{h}}$, for $t\leq s\leq T$,
\begin{equation*}
\begin{aligned} & \ep_{t}\Vert\boldsymbol{X}^{h,m_{h}}(s)-\boldsymbol{X}(s)\Vert^{2}\\
= & \ep_{t}\Big\Vert \int_{t}^{s}\int_{\boldsymbol{U}}\boldsymbol{b}(z,\boldsymbol{X}^{h,m_{h}}(z),m_{h},\boldsymbol{r})\gamma_{z}^{h,m_{h}}(d\boldsymbol{r})dz\\
 & +\int_{t}^{s}\boldsymbol{\sigma}(z,\boldsymbol{X}^{h,m_{h}}(z)) d\boldsymbol{W}(z)+\boldsymbol{\varepsilon}^{h}(s)\Big\Vert ^{2}\\
\leq & Ks^{2}+Ks+\varepsilon^{h}(s),
\end{aligned}
\end{equation*}
where $\limsup_{h\rightarrow0}\ep
|\varepsilon^{h}(s)|\rightarrow0$
and $K$ is a positive constant whose value may be different in different
context. Similarly, we can guarantee $\ep_{t}
\Vert \boldsymbol{X}^{h,m_{h}}(s+\delta)-\boldsymbol{X}^{h,m_{h}}(s)
\Vert ^{2}=O(\delta)+\varepsilon^{h}(\delta)$.
Therefore, the tightness of $\boldsymbol{X}^{h,m_{h}}$ follows. Using
the properties of local consistency, it can be argued (cf. the proof
of [\cite{KhDp:01}, Theorem 9.4.1]) that $\{\boldsymbol{B}^{h,m_{h}}\}_{h>0}$
is tight. Next, we show that the limit is the solution of SDEs driven $\gamma$.

Then by the tightness of $\boldsymbol{X}^{h,m_{h}}$ and $m_{h}$, we have
\begin{equation}
\begin{aligned}\boldsymbol{X}^{h,m_{h}}(s)= & \boldsymbol{x}+\int_{t}^{s}\int_{\boldsymbol{U}}\boldsymbol{b}(z,\boldsymbol{X}^{h,m_{h}}(z),m_{h},\boldsymbol{r})\gamma_{z}^{h,m_{h}}(d\boldsymbol{r})dz\\
 & +\int_{t}^{s}\boldsymbol{\sigma}(z,\boldsymbol{X}^{h,m_{h}}(z)) d\boldsymbol{W}^{h,m_{h}}+\boldsymbol{\varepsilon}^{h,\delta}(s),
\end{aligned}\label{eq:5-2-2-2}
\end{equation}
where $\lim_{\delta\rightarrow0}\limsup_{h\rightarrow0}\ep\Vert \boldsymbol{\varepsilon}^{h,\delta}(s)\Vert \rightarrow0$.

If we can verify $\boldsymbol{W}(\cdot)$ to be an $\mathcal{F}_t$-martingale, where $\mathcal{F}_t:=\sigma\{\boldsymbol{X}^{h,m_{h}}, \boldsymbol{W}^{h,m_{h}}, \gamma^{h,m_{h}}\}$, then (\ref{eq:5-2-2-2}) could be obtained by taking limits in (\ref{eq:5-2-2-2}).
Let $S(\cdot)$ be a real-valued and continuous function of its arguments with compact support, we have
\begin{equation*}
\begin{aligned}
 & \ep S\left(\boldsymbol{X}^{h,m_{h}}(nh_{2}),\boldsymbol{W}^{h,m_{h}}(nh_{2}),\gamma^{h,m_{h}}(nh_{2}),n\leq N\right)\\
 & \quad\times[\boldsymbol{W}^{h,m_{h}}(t+s)-\boldsymbol{W}^{h,m_{h}}(t)]=\boldsymbol{0}.
\end{aligned}
\end{equation*}
By using the Skorohod representation and the dominated
convergence theorem, letting $h \rightarrow 0$, we obtain
\begin{equation}
\begin{aligned}
 & \ep S\left(\boldsymbol{X}^{h,m_{h}}(nh_{2}),\boldsymbol{W}^{h,m_{h}}(nh_{2}),\gamma^{h,m_{h}}(nh_{2}),n\leq N]\right) \\
 & \quad\times[\boldsymbol{W}(t+s)-\boldsymbol{W}(t)]=\boldsymbol{0}.\label{eq:5-2-2-3}
\end{aligned}
\end{equation}
Since $\boldsymbol{W}(\cdot)$ has continuous sample paths, (\ref{eq:5-2-2-3}) implies that $\boldsymbol{W}(\cdot)$ is a continuous $\mathcal{F}_t$-martingale. On the other hand, since
\begin{equation}
\begin{aligned} & \ep[\Vert \boldsymbol{W}^{h,m_{h}}(t+h_{2})\Vert ^{2}-\Vert \boldsymbol{W}^{h,m_{h}}(t)\Vert ^{2}]\\
= & \ep[\Vert \boldsymbol{W}^{h,m_{h}}(t+h_{2})-\boldsymbol{W}^{h,m_{h}}(t)\Vert ^{2}]=h_{2},\label{eq:5-2-2-4}
\end{aligned}
\end{equation}
by using the Skorohod representation and the dominant convergence theorem together with (\ref{eq:5-2-2-4}), we have
\begin{equation*}
\begin{aligned} 
& \ep S\left(\boldsymbol{X}^{h,m_{h}}(nh_{2}),\boldsymbol{W}^{h,m_{h}}(nh_{2}),\gamma^{h,m_{h}}(nh_{2}),n\leq N\right)\\
 & \quad\times[\Vert \boldsymbol{W}^{h,m_{h}}(t+h_{2})\Vert ^{2}-\Vert \boldsymbol{W}^{h,m_{h}}(t)\Vert ^{2}-h_{2}]=0.
\end{aligned}
\end{equation*}
Then $\boldsymbol{W}(\cdot)$ is an $\mathcal{F}_t$-Wiener process.

We further assume that the probability space is chosen as required
by the Skorohod representation. Therefore, we can assume the sequence
\begin{equation*}
(\boldsymbol{X}^{h,m_{h}},\gamma^{h,m_{h}})\rightarrow(\boldsymbol{X},\gamma)\;\text{w.p.1}
\end{equation*}
with a slight abuse of notation. This leads to that as $h\rightarrow0$,
\begin{equation*}
\begin{aligned} & \ep\left\Vert \int_{t}^{s}\int_{\boldsymbol{U}}\boldsymbol{b}(z,\boldsymbol{X}^{h,m_{h}}(z),m_{h},\boldsymbol{r})\gamma_{z}^{h,m_{h}}(d\boldsymbol{r})dz\right.\\
 & \left.-\int_{t}^{s}\int_{\boldsymbol{U}}\boldsymbol{b}(z,\boldsymbol{X}(z),m,\boldsymbol{r})\gamma_{z}^{h,m_{h}}(d\boldsymbol{r})dz\right\Vert \rightarrow0
\end{aligned}
\end{equation*}
uniformly in $s$. Also, recall that $\gamma^{h,m_{h}}\rightarrow\gamma$
in the `compact weak' topology if and only if
\begin{equation*}
\begin{aligned} & \int_{t}^{s}\int_{\boldsymbol{U}}\phi(z,\boldsymbol{X}(z),m,\boldsymbol{r})\gamma_{z}^{h,m_{h}}(d\boldsymbol{r})dz\\
&\rightarrow \int_{t}^{s}\int_{\boldsymbol{U}}\phi(z,\boldsymbol{X}(z),m,\boldsymbol{r})\gamma_{z}(d\boldsymbol{r})dz,
\end{aligned}
\end{equation*}
for any continuous and bounded function $\phi(\cdot)$ with compact
support. Thus, the weak convergence and the Skorohod representation imply
that as $h\rightarrow0$,
\begin{equation*}
\begin{aligned}
 & \int_{t}^{s}\int_{\boldsymbol{U}}\boldsymbol{b}(z,\boldsymbol{X}(z),m,\boldsymbol{r})\gamma_{z}^{h,m_{h}}(d\boldsymbol{r})dz\\
&\rightarrow \int_{t}^{s}\int_{\boldsymbol{U}}\boldsymbol{b}(z,\boldsymbol{X}(z),m,\boldsymbol{r})\gamma_{z}(d\boldsymbol{r})dz
\end{aligned}
\end{equation*}
uniformly in $s$ on any bounded interval w.p.1. Then we have
\begin{equation}
\begin{aligned}
\boldsymbol{X}(s)= & \boldsymbol{x}+\int_{t}^{s}\int_{\boldsymbol{U}}\boldsymbol{b}(z,\boldsymbol{X}(z),m,\boldsymbol{r})\gamma_{z}(d\boldsymbol{r})dz \nonumber\\
 & +\int_{t}^{s}\boldsymbol{\sigma}(z,\boldsymbol{X}(z)) d\boldsymbol{W}(z)+\boldsymbol{\varepsilon}^{\delta}(s),\label{eq:5-2-2-1}
\end{aligned}
\end{equation}
where $\lim_{\delta\rightarrow0}\ep\Vert \boldsymbol{\varepsilon}^{\delta}(s)\Vert =0$.

Combining the tightness of $\{ \boldsymbol{X}^{h,m_{h}}\}$ and the fact that $\Phi^{h}(m_{h})=\mathbb{P}\circ(\hat{\boldsymbol{X}}^{h,m_{h}})^{-1}$
gives the relative compactness of $\{\Phi^{h}(m_{h})\}$
in $\mathcal{P}^{2}(\mathcal{Q})$. Suppose now that along a subsequence
(relabeled again as $\{h\}$)
\begin{equation*}
(\boldsymbol{X}^{h,m_{h}},\boldsymbol{W}^{h,m_{h}},\gamma^{h,m_{h}})\Rightarrow(\boldsymbol{X},\boldsymbol{W},\gamma),
\end{equation*}
where the notation $\Rightarrow$ denotes weak convergence, and
\begin{equation*}
\Phi^{h}(m_{h})\rightarrow m.
\end{equation*}
Then we also have that
\begin{equation*}
m_{h}\rightarrow m.
\end{equation*}
The proof is completed.
\end{proof}

\section{Proof of Theorem \ref{Theorem 6}}\label{Appendix 7.2}
\begin{proof}
We first show that $\lim\inf_{h\rightarrow0}v^{h,m_{h}}(t,\boldsymbol{x})\geq v(t,\boldsymbol{x})$.
Let $\hat{\gamma}^{h,m_{h}}$ be an optimal relaxed control for $\boldsymbol{X}^{h,m_{h}}$
for each $h$. That is,
\begin{equation*}
\begin{aligned}v^{h,m_{h}}(t,\boldsymbol{x}) & =J^{h,m_{h}}(t,\boldsymbol{x},\hat{\gamma}^{h,m_{h}})\\
 & =\inf_{\gamma^{h,m_{h}}\in\Gamma^{h}}J^{h,m_{h}}(t,\boldsymbol{x},\gamma^{h,m_{h}}).
\end{aligned}
\end{equation*}
Choose a subsequence $\{\tilde{h}\}$ of $\{h\}$ such that
\begin{equation*}
\begin{aligned}
\liminf_{h\rightarrow0}v^{h,m_{h}}\left(t,\boldsymbol{x}\right) & =\lim_{\tilde{h}}v^{\tilde{h},m_{\tilde{h}}}\left(t,\boldsymbol{x}\right)\\
 & =\lim_{\tilde{h}}J^{\tilde{h},m_{\tilde{h}}}(t,\boldsymbol{x},\hat{\gamma}^{\tilde{h},m_{\tilde{h}}}).
\end{aligned}
\end{equation*}
 Let $\{\boldsymbol{X}^{\tilde{h},m_{\tilde{h}}},\hat{\gamma}^{\tilde{h},m_{\tilde{h}}}\}$
weakly converge to $\{\boldsymbol{X},\gamma\}$. Otherwise, take a
subsequence of $\{\tilde{h}\}$ to assume its weak limit. By Theorem
\ref{Theorem 5}, Skorohod representation and dominated convergence theorem, we
have
\begin{equation*}	
J^{h,m_{h}}(t,\boldsymbol{x},\hat{\gamma}^{h,m_{h}})\rightarrow J(t,\boldsymbol{x},\gamma)\geq v(t,\boldsymbol{x}).
\end{equation*}
It follows that
\begin{equation*}
\liminf_{h\rightarrow0}v^{h,m_{h}}(t,\boldsymbol{x})\geq v(t,\boldsymbol{x}).
\end{equation*}
Next, we show that $\lim\sup_{h\rightarrow0}v^{h,m_{h}}(t,\boldsymbol{x})\leq v(t,\boldsymbol{x})$.

Given any $\rho>0$, there is a $\delta>0$ that we can approximate
$(\boldsymbol{X},\gamma)$ by $(\boldsymbol{X}^{\delta},\gamma^{\delta})$
satisfying
\begin{equation*}
\begin{aligned}\boldsymbol{X}^{\delta}(s)= & \boldsymbol{x}+\int_{t}^{s}\int_{\boldsymbol{U}}\boldsymbol{b}(z,\boldsymbol{X}^{\delta}(z),m_{\delta},\boldsymbol{r})\gamma_{z}^{\delta,m_{\delta}}(d\boldsymbol{r})dz\\
 & +\int_{t}^{s}\boldsymbol{\sigma}(z,\boldsymbol{X}^{\delta}(z)) d\boldsymbol{W}^{\delta},
\end{aligned}
\end{equation*}
where $\gamma_{z}^{\delta,m_{\delta}}$ is a piecewise constant and
takes finitely many values, and the controls are concentrated on the
points $r_{1},r_{2},\dots,r_{N}$, for all $s$. Let $\hat{\boldsymbol{\alpha}}^{\rho,m_{\rho}}$
be the optimal control and $\gamma^{\rho,m_{\rho}}$ be corresponding
relaxed controls representation, and let $\hat{\boldsymbol{X}}^{\rho,m_{\rho}}$
be the associated solution process. Since $\hat{\gamma}^{\rho,m_{\rho}}$
is optimal in the chosen class of controls, we have
\begin{equation}
J(t,\boldsymbol{x},\hat{\gamma}^{\rho,m_{\rho}})\leq v(t,\boldsymbol{x})+\frac{\rho}{3}.\label{eq:5-2-3-1}
\end{equation}
Note that for each given integer $\kappa$, there is a measurable
function $\varLambda_{\kappa}^{\rho}$ such that
\begin{equation*}
\hat{\boldsymbol{\alpha}}^{\rho,m_{\rho}}(s)=\varLambda_{\kappa}^{\rho}(\boldsymbol{W}_{l,\tilde{l}}(\tilde{s}),\tilde{s}\leq\kappa\delta,l,\tilde{l}\leq N),\;\text{on}\;[\kappa\delta,\kappa\delta+\delta).
\end{equation*}
We next approximate $\varLambda_{\kappa}^{\rho}$ by a function that
depends on the sample of $(\boldsymbol{W}_{l,\tilde{l}},l,\tilde{l}\leq N)$
at a finite number of time points. Let $\varsigma<\delta$ such that
$\delta/\varsigma$ is an integer. Because the $\sigma$-algebra determined
by $\{\boldsymbol{W}_{l,\tilde{l}}(\vartheta\varsigma),\vartheta\varsigma\leq\kappa\delta,l,\tilde{l}\leq N\}$
increases to the $\sigma$-algebra determined by $\{\boldsymbol{W}_{l,\tilde{l}}(\tilde{s}),\tilde{s}\leq\kappa\delta,l,\tilde{l}\leq N\} $,
the martingale convergence theorem implies that for each $\delta$
and $\kappa$, there are measurable functions $\varLambda_{\kappa}^{\rho,\varsigma}$,
such that as $\varsigma\rightarrow0$,
\begin{equation*}
\begin{aligned}\varLambda_{\kappa}^{\rho,\varsigma}(\boldsymbol{W}_{l,\tilde{l}}(\vartheta\varsigma),\vartheta\varsigma\leq\kappa\delta,l,\tilde{l}\leq N) & =\boldsymbol{\alpha}_{l}^{\rho,\varsigma}\rightarrow\hat{\boldsymbol{\alpha}}^{\rho,\varsigma}(\kappa\delta)\;\text{w.p.1}.\end{aligned}
\end{equation*}
Here, we select $\varLambda_{\kappa}^{\rho,\varsigma}$ such that
there are $N$ disjoint hyper-rectangles that cover the range of its
arguments and that $\varLambda_{\kappa}^{\rho,\varsigma}$ is constant
on each hyper-rectangle. Let $\gamma^{\rho,\varsigma,m_{\rho,\varsigma}}$
denote the relaxed controls representation of the ordinary control
$\boldsymbol{\alpha}^{\rho,\varsigma,m_{\rho,\varsigma}}$ which takes
values $\boldsymbol{\alpha}_{l}^{\rho,\varsigma,m_{\rho,\varsigma}}$
on $[\kappa\delta,\kappa\delta+\delta)$, and let $\boldsymbol{X}^{\rho,\varsigma,m_{\rho,\varsigma}}$
denote the associated solution. Then for small enough $\varsigma$,
we have
\begin{equation}
J(t,\boldsymbol{x},\gamma^{\rho,\varsigma,m_{\rho,\varsigma}})\leq J(t,\boldsymbol{x},\hat{\gamma}^{\rho,m_{\rho}})+\frac{\rho}{3}.\label{eq:5-2-3-2}
\end{equation}
Next, we adapt $\varLambda_{\kappa}^{\rho,\varsigma}$ such that it
can be applied to $\boldsymbol{X}_{n}^{h,m_{h}}$. Let controls $\bar{\boldsymbol{\alpha}}_{n}^{h,m_{h}}$
be used for the approximation chain $\boldsymbol{X}_{n}^{h,m_{h}}$.

For $\kappa=1,2,...$ and $n$ such that $nh_{2}\in[\kappa\delta,\kappa\delta+\delta)$,
we use the controls which are defined by $\bar{\boldsymbol{\alpha}}_{n}^{h,m_{h}}=\varLambda_{\kappa}^{\rho,\varsigma}(\boldsymbol{W}_{l,\tilde{l}}^{h,m_{h}}(\vartheta\varsigma),\vartheta\varsigma\leq\kappa\delta,l,\tilde{l}\leq N)$.
Recall that $\bar{\gamma}^{h,m_{h}}$ denotes the relaxed control representation
of the continuous interpolation of $\bar{\boldsymbol{\alpha}}_{n}^{h,m_{h}}$,
then
\begin{equation*}
\begin{aligned} & (\boldsymbol{X}^{h,m_{h}},\bar{\gamma}^{h,m_{h}},\boldsymbol{W}_{l,\tilde{l}}^{h,m_{h}},\\
 & \varLambda_{\kappa}^{\rho,\varsigma}(\boldsymbol{W}_{l,\tilde{l}}^{h,m_{h}}(\vartheta\varsigma),\vartheta\varsigma\leq\kappa\delta,l,\tilde{l}\leq N,\kappa=0,1,2,\ldots))\\
&\rightarrow (\boldsymbol{X}^{\rho,\varsigma,m_{\rho,\varsigma}},\bar{\gamma}^{\rho,\varsigma,m_{\rho,\varsigma}},\\
 & \boldsymbol{W}_{l,\tilde{l}},\varLambda_{\kappa}^{\rho,\varsigma}(\boldsymbol{W}_{l,\tilde{l}}(\vartheta\varsigma),\vartheta\varsigma\leq\kappa\delta,l,\tilde{l}\leq N,\kappa=0,1,2,\ldots)).
\end{aligned}
\end{equation*}
Thus
\begin{equation}
J(t,\boldsymbol{x},\bar{\gamma}^{h,m_{h}})\leq J(t,\boldsymbol{x},\gamma^{\rho,\varsigma,m_{\rho,\varsigma}})+\frac{\rho}{3}.\label{eq:5-2-3-3}
\end{equation}
Note that
\begin{equation*}
v^{h,m_{h}}(t,\boldsymbol{x})\leq J(t,\boldsymbol{x},\bar{\gamma}^{h,m_{h}}).
\end{equation*}
Combing the inequalities (\ref{eq:5-2-3-1}), (\ref{eq:5-2-3-2}) and
(\ref{eq:5-2-3-3}), we have $\lim\sup_{h\rightarrow0}v^{h,m_{h}}(t,\boldsymbol{x})\leq v(t,\boldsymbol{x})$
for the chosen subsequence. According to the tightness of $(\boldsymbol{X}^{h,m_{h}},\bar{\gamma}^{h,m_{h}})$
and arbitrary of $\rho$, we get
\begin{equation*}
\limsup_{h\rightarrow0}v^{h,m_{h}}(t,\boldsymbol{x})\leq v(t,\boldsymbol{x}),
\end{equation*}
and conclude the proof.
\end{proof}

\bibliographystyle{plain}
\bibliography{mybib}

\end{document}